\documentclass[a4paper,11pt,oneside]{article}
\usepackage[a4paper, margin=1in]{geometry}
\usepackage{setspace}
\usepackage{textcomp}
\usepackage{enumerate}
\usepackage{amsmath, amssymb,amsfonts,amsthm,mathrsfs}
\usepackage{cancel, cases}
\usepackage{soul}
\usepackage{float,graphicx,subcaption,epstopdf}
\usepackage{longtable,tabu,threeparttable,rotating,adjustbox}
\usepackage[normalem]{ulem}
\usepackage[dvipsnames]{xcolor}
\usepackage{tcolorbox} 

\usepackage{hyperref}

\usepackage{setspace,caption}

\doublespace
\newcommand\cF{{\cal F}}

\newcommand\tr{\text{tr}}

\newcommand{\bi}{\begin{itemize}}
\newcommand{\ei}{\end{itemize}}

\theoremstyle{plain}
\newtheorem{theorem}{Theorem}
\newtheorem{corollary}[theorem]{Corollary}
\newtheorem{lemma}[theorem]{Lemma}
\newtheorem{proposition}[theorem]{Proposition}

\theoremstyle{definition}
\newtheorem{definition}[theorem]{Definition}
\newtheorem{assumption}[theorem]{Assumption}

\theoremstyle{remark}
\newtheorem{remark}{Remark}
\newtheorem{notation}{Notation}

\numberwithin{equation}{section}
\numberwithin{theorem}{section}

\usepackage[round, authoryear]{natbib}

\title{On the Separation of Estimation and Control in Risk-Sensitive Investment Problems under Incomplete Observation\footnote{The authors wish to express their sincere gratitude to Hideo Nagai for his comments, suggestions, clarifications, and ongoing support.}}

\author{S\'ebastien Lleo\thanks{Finance Department, NEOMA Business School, 59 rue Pierre Taittinger, 51100 Reims, France; Email: sebastien.lleo@NEOMA-bs.fr.} \thanks{Corresponding author.} \and Wolfgang J. Runggaldier\thanks{Departments of Mathematics, University of Padova, Via Trieste 63, 35121 Padova, Italy, and Fellow Institut Louis Bachelier.}}
\date{\today}
\begin{document}
\maketitle

\begin{abstract}
A typical approach to tackle stochastic control problems with partial observation is to separate the control and estimation tasks. However, it is well known that this separation generally fails to deliver an actual optimal solution for risk-sensitive control problems. This paper investigates the separability of a general class of risk-sensitive investment management problems when a finite-dimensional filter exists. We show that the corresponding separated problem, where instead of the unobserved quantities, one considers their conditional filter distribution given the observations, is strictly equivalent to the original control problem. We widen the applicability of the so-called Modified Zakai Equation (MZE) for the study of the separated problem and prove that the MZE simplifies to a PDE in our approach. Furthermore, we derive criteria for separability. We do not solve the separated control problem but note that the existence of a finite-dimensional filter leads to a finite state space for the separated problem. Hence, the difficulty is equivalent to solving a complete observation risk-sensitive problem. Our results have implications for existing risk-sensitive investment management models with partial observations in that they establish their separability. Their implications for future research on new applications is mainly to provide conditions to ensure separability.

\end{abstract}




\section{Introduction}\label{sec:introduction}

Stochastic control problems with partial observation require a joint estimation of the unobservable state of the system and an optimization of the state-dependent control criterion. This joint estimation-optimization problem is particularly challenging. The go-to solution is to separate the problems into two sequential steps: estimate the state variable via filtering and optimize the criterion with the unobservable state replaced by its estimate \citep{wonhamSeparationTheoremStochastic1968,Lindquist1973,fl75,davis77}. However, in so doing, one may lose full optimality. For instance, the class of risk-sensitive control problems is not separable in general \citep{bevs85,Bensoussan1992}. If one still achieves full optimality we say that a (strict) separation property holds. For example, a few risk-sensitive control problems are separable \citep{na00,nape02} or appear separable \citep{dall_BLcontinuous,dall_BBL_2020,davisRisksensitiveBenchmarkedAsset2021}.

A more complete approach is to pass to the so-called \textit{separated problem} where, in place of the unobserved state variable, one considers its conditional or filter distribution, given the available observations. Yet, little care is usually given to showing that, on the basis of the information provided by the available observations, the so-obtained separated problem is indeed equivalent to the original one under partial observations. In this paper, we deal with this issue as one of our objectives. Even if equivalence holds, there is also the problem that whereas the unobserved state is typically finite-dimensional, the filter distribution is, in general, infinite-dimensional. In a wider sense, one can then consider a separation property to hold if, in the separated problem, the filter distribution can be parametrized by a finite-dimensional parameter process. This is particularly relevant in view of solving the control part of the problem because then these parameters become a finite-dimensional state variable process that can be considered observable because it can be computed on the basis of the observations. Such state variables typically satisfy stochastic dynamic equations. So the control part of the problem can be approached by standard methods for full observation problems, such as an HJB approach.

An alternative approach to solve partially observed control problems, in particular of the risk-sensitive type, results from the work of Nagai and Peng (see \citeauthor{na00}, \citeyear{na00} and \citeauthor{nape02}, \citeyear{nape02}). They introduce what they call a \textit{Modified Zakai Equation (MZE)}. This MZE turns out to be a useful tool since the objective function of the control problem can be expressed in terms of the solution to the MZE.  However, this MZE is a stochastic PDE, so obtaining an explicit solution is generally difficult. Nagai and Peng succeed in deriving an explicit solution for the case of a linear-Gaussian model, thereby also obtaining strict separability. We shall show that in our situation, the MZE reduces to a deterministic PDE.

In the given context, and inspired by the work of Nagai and Peng, the objective of our study is to consider more general nonlinear models for which a finite-dimensional filter exists, characterized by a filter parameter process ${\zeta}_t$. The best-known finite-dimensional filters exist for linear-Gaussian models (Kalman-Bucy filter) and for finite-state Markov processes \citep[for general results, in this case, see][]{elliottNewFinitedimensionalFilters1993}. Other classical finite-dimensional filters are those in \citet{benesExactFinitedimensionalFilters1981} and \citet{daumExactFinitedimensionalNonlinear1986}. Since exact finite-dimensional filters are difficult to obtain, finite-dimensional approximations have been extensively studied, starting from the Extended Kalman Filter [EKF] obtained via a linearization approach. Other approaches rely on Markov chain approximations \citep{kudu01}, projections on a finite-dimensional manifold \citep{brigoApproximateNonlinearFiltering1999}, and Fourier expansions \citep{Mikulevicius1993}. The finite-dimensionality of filters has been studied in particular in discrete-time models; such a model may also result from a time discretization of a continuous-time model (actual observations are naturally mostly in discrete time), see Section 2.3. in \citet{runggaldierApproximationsDiscreteTime1994} and \citet{runggaldierSufficientConditionsFinite2001}.   

Our aim, which also represents our contribution with respect to the literature, in particular \citet{na00} and \citet{nape02},  is to:
\begin{enumerate}
    \item[i)] show the equivalence of the original and separated problems;
    \item[ii)] find a criterion to analyze under which conditions and in what form a separation property holds between estimation and control; and
    \item[iii)] widen the applicability of the MZE approach by deriving an MZE that applies whenever a finite-dimensional filter exists.
\end{enumerate}
Furthermore, we derive a generalized Kallianpur-Striebel formula for the problem in its original version, and we also establish the separability of existing risk-sensitive investment management models.

Concerning the separated problem, given our assumption on the existence of a finite-dimensional filter process ${\zeta}_t$, we shall consider this process ${\zeta}_t$ as the state variable process for the separated problem. We then perform a measure change in the spirit of \citet{kuna02} to express the thus obtained separated risk-sensitive control problem in standard form.

As in the work of Nagai and Peng, our paper considers a partially observed stochastic control problem of the risk-sensitive type that arises from investment in financial markets. A peculiarity of such problems with respect to more engineering-type control problems under partial observation is that the control, which is given by the investment strategy, does not affect the observations. This has an implication also for the issue of separability \citep[in this context see e.g.][]{georgiouSeparationPrincipleStochastic2013}. While the issue of separability for investment problems with a risk-sensitive criterion also appears in \citet{na00} and \citet{nape02}, those studies consider only linear-Gaussian models. Our work generalizes such a setting by allowing for more general models for which we assume the existence of a finite-dimensional filter, parametrized by ${\zeta}_t$. 

Specifically, the set of risk-sensitive control problems we consider in this paper generalizes the setup in \citet{davisRisksensitiveBenchmarkedAsset2021}.  This setup provides a convenient starting point because it already includes most risk-sensitive asset management models \citep{bipl99,na00,kuna02,dall_RSBench}, contains a subclass of Linear-Exponential-Gaussian (LEG) models \citep{ja73,Bensoussan1992}, and because it also connects to the recent literature on optimal investment with expert opinions \citep{fgw12, dall_BLcontinuous, Gabihetal2014, Sass_etal_2017, sass2020diffusion,dall_BBL_2020} which relies extensively on filtering. We extend this setup in two main ways. We remove the linearity assumption to consider general nonlinear dynamics and introduce a mix of observable and unobservable state variables. The reader will find a detailed correspondence between our approach and existing risk-sensitive investment management models in Appendix \ref{app:correspondence}.

Concerning the MZE in our setup, it will be a deterministic PDE, contrary to Nagai and Peng's stochastic PDE. Its solution allows us to determine the value of the objective function of the control problem for each given strategy. We do not discuss how to obtain its solution, but it is a deterministic PDE, so one ends up with the same degree of difficulty as when solving an HJB equation for a complete observation stochastic control problem.

Our paper is organized as follows. Section \ref{sec:setting} describes the model setup. Section \ref{sec:filtering} introduces the finite-dimensional filter, and Section \ref{sec:controlsetting} presents the setup of our risk-sensitive stochastic control problem. Subsection \ref{S.4.1} recalls the main features of the problem in the pre-filter setting with the unobserved state/factor process $X_t$.   Subsection \ref{S.4.2} then presents the setup of the separated problem based on the filter parameter process ${\zeta}_t$ and  in Subsection \ref{S.4.3} we prove the equivalence of the two problem settings, the original one of Subsection \ref{S.4.1}, and the separated one of Subsection \ref{S.4.2}; the proof itself can be found in Appendix \ref{app:proofs:controlpb:equival}. Section \ref{sec:Zakai} then concerns the modified Zakai equation (MZE) for the separated problem. We also propose a generalized Kallianpur-Striebel formula, with the proof in Appendix \ref{app:proof:GKSF}, and establish a relation between the original and separated problems. With this last objective in mind, we summarize the key aspects of the MZE for the problem in its original form, due to Nagai and Peng, in Appendix \ref{app:MZE:NagaiPeng}. Section \ref{sec:Separability} outlines criteria for separability. Finally, Section \ref{sec:examples} provides examples of implementation.

\section{Model Setting}\label{sec:setting}

Let $\left(\Omega,  \mathcal{F}, \left(\mathcal{F}\right)_{t = 0}^T, \mathbb{P} \right)$ be a filtered complete probability space. We consider on this space
a $\mathbb{R}^d$-valued $\mathcal{F}_t$-Wiener process $W_t$ with components $W_t(i), i = 1, \ldots d$, and where, for $k, \ell \geq 0$, $n, m \geq 1$ to be specified below, $d := \ell + n + m + k + 1$.

\subsection{Factor Process}\label{sec:setting:factor}

In our model, asset and benchmark drifts are nonlinear functions of $n$ factors. These factors can be macroeconomic variables (GDP, inflation, interest rates,...). They can also be returns driven by empirical asset pricing factors (market risk premium, value premium, momentum,...). Finally, these factors can be latent variables (obtained, for example, from a principal component analysis). 

Similarly to \citet{PlatenRunggaldier2004,PlatenRunggaldier2007}, we assume that $n \geq 1 $ factors are unobservable while $\ell \geq 0$ factors are observable. The unobservable factors $X$ and observable factors $F$ have the following dynamics:
\begin{numcases}{}
dX_t = b(t,X_t,F_t)dt + \Lambda(t,X_t,F_t) dW_t,    \quad X_0 \sim \mathcal{D}(\mu)    
                                    \label{eq:state:X}\\
dF_t = b^f(t,X_t,F_t)dt + \Lambda^f(t,F_t) dW_t,    \quad F_0 = f_0,
                                    \label{eq:state:F}
\end{numcases}
where $b(t,x,f): [0,T] \times \mathbb{R}^n \times \mathbb{R}^\ell \to \mathbb{R}^n$, $b^f(t,x,f) : [0,T] \times \mathbb{R}^n \times \mathbb{R}^\ell \to \mathbb{R}^\ell$, and the matrix-valued functions $\Lambda(t,x,f): [0,T] \times \mathbb{R}^n \times \mathbb{R}^\ell \to \mathbb{R}^{n \times d}$, and $\Lambda^f(t,f): [0,T] \times \mathbb{R}^\ell \to \mathbb{R}^{\ell \times d}$ are suitable $C^{1,1,1}_b$ or $C^{1,1}_b$ functions that ensure the existence of a strong solution to \eqref{eq:state:X}-\eqref{eq:state:F}.  The random initial value of the unobservable factors, $X_0$ follows a known distribution $\mathcal{D}$ with parameters $\mu$. The initial value $X_0$ is also independent of the Wiener process $W_t$.

\begin{remark}
As customary in filtering, the diffusion coefficient of an observation process cannot depend on the unobservable variable. Such dependence generates a noise-free observation via the quadratic variation of the process and thus causes the filter to degenerate.       
\end{remark}

\subsection{Financial Assets and Benchmark}\label{sec:setting:assets}

The financial market consists of $m$ risky financial assets. Their \emph{discounted} prices follow a geometric dynamics: 
\begin{align}\label{eq:dS}
     \frac{dS^i_t}{S^i_t} 
    &= a_{i}(t,X_t,F_t)dt
    + \sum_{j=1}^{d} \sigma_{ij}\left(t,F_t\right) dW^j_t,
    \qquad
    S_i(0) = s_i,
    i = 1, \ldots, m,
\end{align}
where $a(t,x,f) : [0,T] \times \mathbb{R}^n \times \mathbb{R}^\ell \to \mathbb{R}^m$ and $\Sigma\left(t,f\right) = \left( \sigma_{ij}\left(t,f\right)\right), i = 1, \ldots, m; j = 1, \ldots, d : [0,T] \times \mathbb{R}^{\ell} \to \mathbb{R}^{m \times d}$ are suitable $C^{1,1,1}_b$ or $C^{1,1}_b$ functions that ensures the existence of a strong solution to \eqref{eq:dS}. We also assume that no two assets have an identical risk profile:   
\begin{assumption}\label{as:sigma:posdef}
    The matrix $\Sigma\Sigma'\left(t,F_t\right)$ is positive definite.
\end{assumption}

To tidy the notation in Assumption \ref{as:sigma:posdef}, we denoted the matrix multiplication $\Sigma\left(t,F_t\right)\Sigma'\left(t,F_t\right)$ by $\Sigma\Sigma'\left(t,F_t\right)$. We adopt similar notational shortcuts throughout the paper. 

As in \citet{dall_BBL_2020,davisRisksensitiveBenchmarkedAsset2021}, we distinguish between tradable and non-tradable assets. The investor is allowed to trade an investment universe of $0 < m_1 \leq m$ assets, but not the remaining $m_2 = m - m_1 \geq 0$ assets. Accordingly, we express the securities price vector as $S_t := \begin{pmatrix} S^{(1)'}_t    & S^{(2)'}_t \end{pmatrix}'$, where $S^{(1)}_t$ is the $m_1$-vector process of tradable securities prices and $S^{(2)}$ is the $m_2$-vector process of untradable, but observable, securities prices. We perform a similar decomposition for the vector- and matrix-valued functions $a$ and $\Sigma$ and introduce the following 
\begin{notation}\label{N.1}
Denote by $a^{(1)}$ and $\Sigma^{(1)}$  the subvector and submatrix corresponding to the $m_1$ tradable assets, with an analogous definition for $a^{(2)}$ and $\Sigma^{(2)}.$ \end{notation}

The distinction between tradable and non-tradable assets introduces a constraint in the stochastic control problem, but it does not affect the filtering problem. Investors observe the price trajectory of all the assets to estimate the unobservable factors.


The investor manages a portfolio of financial assets against a benchmark index, typically a financial index or a custom-built passive portfolio. We model the benchmark's \emph{discounted} level as:
\begin{align}\label{eq:dL}
\frac{dL_t}{L_t} &= c(t,X_t, F_t)dt + \Xi\left(t,F_t\right) dW_t,
\quad L(0) = l,
\end{align}
where $c(t,x, f) : [0,T] \times \mathbb{R}^n \times \mathbb{R}^\ell \to \mathbb{R}$ and $\Xi\left(t,f\right): [0,T] \times \mathbb{R}^{\ell} \to \mathbb{R}^{d}$ are suitable $C^{1,1,1}_b$ or $C^{1,1}_b$ functions that ensure the existence of a strong solution.

\subsection{Observation Process}\label{sec:setting:observation}

The investor infers the value of the $n$ unobservable factors $X_t$ from information contained in the observable factors, financial asset prices, benchmark value, and in non-market observations such as expert forecasts and alternative data. We model these observations as the $\mathbb{R}^{m^Y}$-valued process $Y_t := \begin{pmatrix} F_t', \ln S_t', \ln L_t', E_t' \end{pmatrix}'$, where
\begin{enumerate}[i)]
    \item $\ln S_t$ is the $\mathbb{R}^m$-valued vector with elements equal to the logarithm of the  discounted securities prices, that is $\ln S_t^i, i = 1, \ldots, m$.  This vector process solves the SDE
    \begin{align}\label{eq:dfrakS}
            d \ln S_t 
        &= \left[a(t, X_t, F_t) -\frac{1}{2} d_{\Sigma}\left(t,F_t\right) \right]dt
        	+ \Sigma\left(t,F_t\right) dW_t,
            \quad \ln S_0 = \ln s,
    \end{align}
    where $d_{\Sigma}\left(t,F_t\right)$ is the vector containing the elements on the main diagonal of the square matrix $\Sigma\Sigma'\left(t,F_t\right)$, that is, $d_{\Sigma}\left(t,f\right) := \begin{pmatrix} \left( \Sigma\Sigma'\right)_{11}\left(t,f\right) & \left(\Sigma\Sigma'\right)_{22}\left(t,f\right) & \ldots & \left(\Sigma\Sigma' \right)_{mm}\left(t,f\right) \end{pmatrix}'$;

    \item $\ln L_t$ is the logarithm of the discounted benchmark level, with dynamics 
    \begin{align}\label{eq:dfrakL}
        d \ln L_t 
        &= \left[c(t, X_t, F_t) -\frac{1}{2} \Xi\Xi'\left(t,F_t\right)\right]dt
        	+ \Xi\left(t,F_t\right) dW_t,
        \quad \ln L_0 = \ln l;
    \end{align}

    \item  $E_t$ models $k$ expert forecasts. It solves the SDE:  
    \begin{align}\label{eq:dZ}
        dE_t = a^{E}(t,X_t,F_t)dt + \Sigma^{E}\left(t,F_t\right) dW_t,
        \quad E_0 = e_0,
    \end{align}
    where $a^{E}(t,x,f) : [0,T] \times \mathbb{R}^n \times \mathbb{R}^\ell \to \mathbb{R}^k$ and $\Sigma^{E}\left(t,f\right): [0,T] \times \mathbb{R}^{\ell} \to \mathbb{R}^{k \times d}$ are suitable $C^{1,1,1}_b$ and $C^{1,1}_b$ functions that ensure the existence of a strong solution.
\end{enumerate}
Hence, $m^Y := \ell + m + k + 1$. This also justifies the fact of having put $d = \ell + n + m + k + 1$ at the beginning of this section. We express the dynamics of $Y_t$ succinctly as: 
\begin{align}\label{eq:obs}
dY_t &= a^Y(t,X_t,F_t)dt + \Sigma^Y\left(t,F_t\right) dW_t, \; Y_0 = y_0,
\end{align}
where
\begin{align}
a^Y(t,X_t,F_t) &:= \left(
b^f(t,X_t,F_t)',
a(t,X_t,F_t)' - \frac{1}{2}d_{\Sigma}\left(t,F_t\right), 
c(t,X_t,F_t)' - \frac{1}{2} \Xi\Xi'\left(t,F_t\right), 
a^{E}(t,X_t,F_t)' 
\right)',
                        \label{eq:Y:coef:decomp:a}    \\
\Sigma^Y\left(t,F_t\right) &:= \begin{pmatrix} 
\Lambda^f\left(t,F_t\right)', &
\Sigma\left(t,F_t\right)', &  
\Xi\left(t,F_t\right)', & 
\Sigma^{E}\left(t,F_t\right)'
\end{pmatrix}'.
                                \nonumber
\end{align}

\subsection{Summary of the Model}\label{sec:setting:summary}

Table \ref{table:setting:summary} presents the essential information related to the variables used in our model.

\begin{table}[h!]
\begin{center}
\resizebox{1.00\textwidth}{!}{%
\begin{tabular}{| c| l | c | c | c | }
\hline
      Variable
    & Interpretation
    & Dimension
    & Observable?
    & Traded by 
    \\
    &
    &
    &
    & the investor?
    \\
\hline\hline
      $X_t$
    & Factors (hidden) 
    & $n  \geq 1$
    & No
    & No
\\
\hline
      $F_t$
    & Factors (observable) 
    & $\ell \geq 0$
    & Yes
    & No
\\
\hline
      $S_t = \begin{pmatrix}
          S_t^{(1)'} & S_t^{(2)'}
      \end{pmatrix}$
    & Securities 
    & $m = m_1 + m_2$
    & Yes
    & See below
\\
      \hspace{3mm} $S_t^{(1)'}$
    & \hspace{3mm} Tradable securities  
    & $m_1 \geq 1$
    & Yes
    & Yes
\\
      \hspace{3mm} $S_t^{(2)'}$ 
    & \hspace{3mm} Untradable securities 
    & $m_2 \geq 0$
    & Yes
    & No
\\
\hline
      $L_t$
    & Investment benchmark 
    & 1
    & Yes
    & No
\\
\hline
      $E_t$
    & Expert forecasts 
    & $k \geq 0$
    & Yes
    & No
\\
\hline
\end{tabular}%
}
\end{center}
\caption{Summary of the model.}
\label{table:setting:summary}
\end{table}

\section{Filter Setup}\label{sec:filtering}

Let $\mathcal{F}^Y_t=\sigma\{Y(u),0\leq u \leq t\}$ be the filtration generated by the observation process\footnote{By a slight abuse of notation, we refer to $\mathcal{F}^Y_t$ as the \emph{filtration} to identify the filtration $\left( \mathcal{F}^Y \right)_{t=0}^T$ -- by contrast with $\left( \mathcal{F} \right)_{t=0}^T$ -- and point to the $\sigma$-algebra $\mathcal{F}^Y_t$ which tracks the information available at time $t$.}. We embed our partial observation problem in a slightly more general class, where the coefficients may depend on the entire observation process $Y_t$. Now we express $b$, $a^Y$, $\Lambda$, and $\Sigma^Y$ as functions of $Y_t$ instead of just $F_t$, with a slight abuse of notation. Thus, we have
\begin{align}\label{eq:POpb}
\begin{cases}
dX_t &= b(t,X_t,Y_t)dt + \Lambda(t,X_t,Y_t) dW_t,    \quad X_0 \sim D(\mu)
                                    \\
dY_t &= a^Y(t,X_t,Y_t)dt + \Sigma^Y\left(t,Y_t\right) dW_t, \; Y_0 = y_0.
\end{cases}   
\end{align}

We shall base ourselves on the following:
\begin{assumption}\label{as:FDfilter}
There exists a finite-dimensional filter of $X_t$ given $\mathcal{F}^Y_t$, parametrized by $\zeta = \left( \zeta^1, \zeta^2, \ldots, \zeta^q \right)$ for some $q \in \mathbb{N}$, i.e.
\begin{align}\label{eq:proba:FDfilter}
    p_{X_t \mid \mathcal{F}^Y_t} \left(X\right) \sim p\left(X;\zeta\right).
\end{align}
\end{assumption}

Let $ Z := \left(Y_1, \ldots, Y_{m^Y}, \zeta^1, \ldots, \zeta^q \right)'$. Since our model is a diffusion-type model, analogously to what is the case with the Kalman-Bucy filter, we may reasonably assume that $\zeta$ can be represented as a $\mathbb{R}^q$-valued diffusion process with dynamics
\begin{align}\label{eq:zeta}
    d \zeta_t = G(t,Z_t)dt + H(t, Z_t) dY_t,
\end{align}
with initial value $\zeta_0$, and for suitable functions $G$ and $H$ such that a solution exists.

\begin{remark}\label{Wonham}
The dynamics at \eqref{eq:zeta} includes the known finite-dimensional filter dynamics for those cases where the not directly observable process $X_t$ follows a diffusion dynamics as in \eqref{eq:POpb}. It includes also the classical Wonham filter \citep[see][]{wonhamApplicationsStochasticDifferential1964}, where $X_t$ is a finite-state Markov process with states $X_t\in \{x^1,\cdots,x^k\}$ for some $k\in \mathbb{N}$ rather than a diffusion. This process $X_t$, or a function thereof, is assumed to be observed in additive Gaussian noise. More precisely let, for a measurable function $f(\cdot)$ and a constant $\sigma$, the observations be given by $Y_t$ with
$$dY_t=f(X_t)\,dt+\sigma\,d\beta_t$$
where $\beta_t$ is a Wiener process. In this case the filter parameters $\zeta_t$ correspond to the conditional state probabilities $p_t=(p_t^1,\cdots,p_t^k)$ with $p_t^i=P\{X_t=x^i\mid{\cal F}_t^Y\}$ that evolve on a simplex. Even so, recalling that $Z=(Y,\zeta)$ and denoting by $Q=(q_{ij})\in{\mathbb{R}}^{k\times k}$ the generator of $X_t$, they satisfy a dynamics of the form \eqref{eq:zeta} given by (see e.g. equation (2) in \citeauthor{zhangTwotimescaleApproximationWonham2007},\citeyear{zhangTwotimescaleApproximationWonham2007})
$$dp_t=p_tQ\,dt-\frac{1}{\sigma^2}\hat f_tp_tF_tdt+\frac{1}{\sigma^2}p_tF_tdY_t$$
where $\hat f_t=\sum_{i=1}^kf(i)p_t^i$ and $F_t = \text{diag}(f(1),\cdots,f(k))-\hat f_tI$ with $I$ the identity matrix. 
\end{remark}

Next, for a generic integrable function $f\left(t,X,Y\right)$, define
\begin{align}\label{eq:fhat}
    \hat{f}(t,Z) 
    &:= \int  f\left(t,x,Y\right) d p\left(x;\zeta\right)
\end{align}
In particular, we let 
\begin{align}\label{eq:ahat}
    \hat{a}^Y(t,Z) := \int  a^Y\left(t,X,Y\right) d p\left(X;\zeta\right).
\end{align}
Based on \eqref{eq:Y:coef:decomp:a} we decompose $\hat{a}^Y$ as
\begin{align}\label{eq:Y:coef:decomp:tilde}
\hat{a}^Y(t,Z_t) := \begin{pmatrix} 
    \hat{b}^{f}(t,Z_t)', &
    \hat{a}(t,Z_t)' - \frac{1}{2}d_{\Sigma}\left(t, Y_t \right), &
    \hat{c}(t,Z_t) - \frac{1}{2} \Xi\Xi'\left(t, Y_t \right), & 
    \hat{a}^{E}(t,Z_t)' 
\end{pmatrix}'
\end{align}
and consider the further decomposition $\hat{a}(t,Z_t) = \begin{pmatrix} \hat{a}^{(1)'}(t,Z_t) & \hat{a}^{(2)'}(t,Z_t) \end{pmatrix}'$.

Since the accessible filtration is the subfiltration $\mathcal{F}^Y_t \subsetneq \mathcal{F}_t$, we shall consider dynamics restricted to $\mathcal{F}^Y_t$. 

\begin{lemma}[``Innovations Lemma'']\label{lem:Wiener:restricted}

There exists a $\left( \mathbb{P}, \mathcal{F}^Y_t \right)$ $d$-dimensional standard Wiener process $\tilde{W}_t$ such that
\begin{align}\label{eq:lem:Wiener:restricted}
    \Sigma^Y\left(t, Y_t\right) d\tilde{W}_t 
    = \left(a^Y(t,X_t,Y_t) - \hat{a}^Y(t,Z_t) \right)dt + \Sigma^Y\left(t, Y_t\right) dW_t.
\end{align}

\end{lemma}

For the proof, see Theorem 7.12 in \cite{LiShI04} or the Appendix in \cite{PlatenRunggaldier2004}.

\begin{remark}\label{rk:innovation}
Notice that the $d$-dimensional ($d=m^Y+n$) Wiener process $\tilde{W}_t$ in the above lemma is closely related but not equal to the traditional innovations process $U_t$ as it appears in the Kalman-Bucy filter setup and where it should have dimension $m^Y$. Following from equation \eqref{eq:lem:Wiener:restricted}, we have
\begin{align}\label{eq:rk:innovation:Kalmaninno} 
    dU_t 
    := \left( \Sigma^{Y}\Sigma^{Y'}\left(t, Y_t\right)\right)^{-\frac{1}{2}} \left(
    dY_t - \hat{a}^Y(t,Z_t) dt \right)
    = \left( \Sigma^{Y}\Sigma^{Y'}\left(t, Y_t\right)\right)^{-\frac{1}{2}}\Sigma^{Y}\left(t, Y_t\right) d\tilde{W}_t,
\end{align}
where all the quantities are on $\mathcal{F}^Y_t$. We nevertheless called the lemma ``innovations lemma'' since our $\tilde{W}_t$ serves the same purpose as the traditional innovations process, namely to standardize the difference between the actual observation drift $a^Y(t,X_t,Y_t)$ which, contrary to the Kalman filter setting, is here in general nonlinear, and its expectation is $\hat a^Y(t,Z_t)$. Furthermore, although $\tilde{W}_t$ is a translation of $W_t$, it is not a translation coming from a Girsanov measure transformation. Both $W_t$ and $\tilde{W}_t$ are Wiener processes under $\mathbb{P}$. However, $W_t$ is defined on the filtration $\cF_t$, while $\tilde{W}_t$ is defined on  $\cF^Y_t \subsetneq \cF_t$.  
\end{remark}

\begin{remark}\label{rk:Wtilde:blocks}
In block matrix form, equation \eqref{eq:lem:Wiener:restricted} reads
{\small
\begin{align}\label{eq:lem:Wiener:restricted:blockmat}
    \begin{pmatrix} 
        \Lambda^f\left(t,Y_t\right) \\
        \Sigma\left(t,Y_t\right) \\  
        \Xi\left(t,Y_t\right) \\ 
        \Sigma^{E}\left(t,Y_t\right)
    \end{pmatrix}
    d\tilde{W}_t 
    = 
    \left[
    \begin{pmatrix} 
        b^{f}(t,Y_t) \\
        a(t,Y_t) - \frac{1}{2}d_{\Sigma}\left(t, Y_t \right) \\
        c(t,Z_t) - \frac{1}{2} \Xi\Xi'\left(t, Y_t \right) \\ 
        a^{E}(t,Y_t) 
    \end{pmatrix}  
    - 
    \begin{pmatrix} 
        \hat{b}^{f}(t,Y_t) \\
        \hat{a}(t,Y_t) - \frac{1}{2}d_{\Sigma}\left(t, Y_t \right) \\
        \hat{c}(t,Y_t) - \frac{1}{2} \Xi\Xi'\left(t, Y_t \right) \\ 
        \hat{a}^{E}(t,Y_t) 
    \end{pmatrix} 
    \right]dt 
    +
    \begin{pmatrix} 
        \Lambda^f\left(t,Y_t\right) \\
        \Sigma\left(t,Y_t\right) \\  
        \Xi\left(t,Y_t\right) \\ 
        \Sigma^{E}\left(t,Y_t\right)
    \end{pmatrix}
    dW_t.
\end{align}
}
In particular, the second and third rows of the block matrix imply that
{\small
\begin{align}
    \Sigma\left(t,Y_t\right) d\tilde{W}_t 
    = \left( a(t,Y_t) - \hat{a}(t,Y_t)\right) + \Sigma\left(t,Y_t\right) dW_t
    \qquad
    \Xi\left(t,Y_t\right) d\tilde{W}_t 
    = \left( c(t,Y_t) - \hat{c}(t,Y_t)\right) + \Xi\left(t,Y_t\right) dW_t.
                                            \nonumber
\end{align}
}
We use these relations throughout the paper to ensure that individual component processes $F_t, S_t, L_t, E_t$ of the observation $Y_t$ are well defined in terms of subvectors of $\Sigma^Y\left(t, Y_t\right) d\tilde{W}_t$. 
\end{remark}

Applying Lemma \ref{lem:Wiener:restricted}, we express \eqref{eq:obs} and \eqref{eq:zeta} as: 
\begin{numcases}{}
dY_t = \hat{a}^Y(t, Z_t) dt + \Sigma^Y\left(t, Y_t\right) d\tilde{W}_t,
                                    \label{eq:obs:tilde} \\
d\zeta_t = \hat{G}(t, Z_t)dt + H(t, Z_t)\Sigma^Y\left(t, Y_t\right) d\tilde{W}_t,
                                    \label{eq:zeta:wtilde}
\end{numcases}
where $\hat{G}(t, Z_t) := G(t, Z_t) + H(t, Z_t)\hat{a}^Y(t, Z_t).$

Furthermore, using the definition of $Y_t$ and  the decomposition \eqref{eq:Y:coef:decomp:tilde}, we have
\begin{align}
\frac{dS_t^{i}}{S_t^{i}}
    = \hat{a}_{i}(t,Z_t)dt
    + \sum_{j=1}^{d} \sigma_{ij}\left(t, Y_t\right) d\tilde{W}_{j}(t),
&& \frac{dL_t}{L_t} = \hat{c}(t,Z_t)dt 
    + \Xi\left(t, Y_t\right) d\tilde{W}_t.
\end{align}


\section{Risk-Sensitive Control Problem}\label{sec:controlsetting}

The control problem that we consider is a partial observation control problem, where the information to an agent is represented by the observation filtration $\mathcal{F}^Y_t.$ It becomes thus natural to look at a problem formulation that is restricted to $\mathcal{F}^Y_t$ by using the filter for the unobserved process $X_t$. Still, in line with \citet{na00} and \citet{nape02}, and in view also of later comparisons, we shall first present the risk-sensitive control problem expressed in terms of the unobserved factor process $X_t$, i.e. before applying the filter. We shall call this problem the \textit{original problem}. We shall then describe the problem in its so-called  \textit{ separated form,} expressed in terms of the parameters of the filter that, according to Assumption \ref{as:FDfilter}, are supposed to be finite-dimensional.
We shall mainly concentrate on this second version of the problem as it is where we make our novel contribution. 

The investment strategy is a main ingredient for both versions of the problem. More precisely, we let the investment strategy $h_t$  be the $m_1$-element vector process representing the proportion of wealth invested in the $m_1$ tradable financial securities. The investment strategy has to be based on the available information that is given by the observations $Y_t$, and so we shall consider it to be $\mathcal{F}^Y_t$-adapted for both versions of the problem in Subsections \ref{S.4.1} and \ref{S.4.2}

\subsection{The problem expressed in terms of the unobserved factor process $X_t$}\label{S.4.1}
The discounted wealth process $V_t$ is the market value of the self-financing investment portfolio subject to the investment strategy $h_t$. It solves the SDE: 
\begin{align}\label{eq:V}
\frac{dV_t}{V_t}
&= \displaystyle{\sum_{i=1}^{m_1} h_t^i\,\frac{dS_t^i}{S_t^i}}= h_t' a^{(1)}(t,X_t, Y_t) dt
+ h_t' \Sigma^{(1)} \left(t, Y_t\right)	 dW_t,
\qquad
V_0 = v.
\end{align}
where $a^{(1)}$ and $\Sigma^{(1)}$ refer to the drift and diffusion of the tradable assets, as introduced in Notation \ref{N.1}. The log excess return $R_t := \ln \frac{V_t}{L_t}$ tracks the portfolio's performance relative to its benchmark. Its dynamics is:
\begin{align}\label{eq:R}
dR_t &= \left[ 
  \left(- \frac{1}{2} h_t'\Sigma^{(1)}\Sigma^{(1)'} \left(t, Y_t\right) h_t + h_t'a^{(1)}(t,X_t, Y_t) + \frac{1}{2}\Xi\Xi'\left(t, Y_t\right) - c(t,X_t, Y_t \right) \right] dt
                                    \nonumber\\
    & + \left(h_t'\Sigma^{(1)}\left(t, Y_t\right) - \Xi\left(t, Y_t\right) \right) dW_t,
\qquad
R_0 = \ln \frac{v}{l} =: r_0.     
\end{align}

The risk-sensitive benchmarked criterion $J$ is then
\begin{align}\label{eq:J}
J(h;T,\theta,r_0) 
&:= -\frac{1}{\theta} \ln \mathbf{E} \left[ r_0^{-\theta} e^{-\theta R_T}\right]
= -\frac{1}{\theta} \ln \mathbf{E} \left[    
r_0^{-\theta} \exp \left\{ \theta \int_{0}^{T} g(t,X_t, Y_t,h_t;\theta) dt \right\} \chi_T^h
\right],
\end{align}
where $\theta \in (-1,0)\cup(0,\infty)$ is the risk sensitivity parameter, $T < \infty$ is a fixed time horizon,

\begin{align}\label{eq:g}
g(t,x,y,h;\theta)
&:= \frac{\theta+1}{2}h'\Sigma^{(1)}\Sigma^{(1)'} \left(t, y\right) h 
- h' a^{(1)}(t,x,y)
- \theta h'\Sigma^{(1)}\Xi'\left(t,y\right)
+ c(t,x,y)
                                        \nonumber\\
&+ \frac{\theta -1}{2}\Xi\Xi'\left(t,y\right),
\end{align}
and the Dol\'eans-Dade exponential $\chi_t^h, \; t \in [0,T]$ is defined via  
\begin{align}\label{eq:expmartX}
    \chi_t^h 
    &:= \exp \bigg\{ -\theta \int_{0}^{t}\left(h_s' \Sigma^{(1)}\left(s, Y_s\right) - \Xi\left(s, Y_s\right) \right)d W_s  
                    \nonumber\\  
    &
        -\frac{1}{2} \theta^2     \int_{0}^{t} \left(h_s' \Sigma^{(1)}\left(s, Y_s\right) - \Xi\left(s, Y_s\right) \right)
        \left(\Sigma^{(1)'}\left(s, Y_s\right) h_s - \Xi'\left(s, Y_s\right) \right)ds 
    \bigg\}.
\end{align}

\begin{remark}
The only case of practical interest is $\theta \in (0,\infty)$. It corresponds to investors with a higher risk aversion than the Kelly/log-utility investor. The case $\theta \in (-1,0)$ leads to overbetting strategies that leverage the Kelly portfolio. These strategies effectively trade expected returns for more volatility  \citep{davisRisksensitiveBenchmarkedAsset2021}. So one should avoid these strategies strictly. Additionally, the choice of model parameters is crucial when $\theta \in (-1,0)$ to ensure the investment problem is well-posed. \citet{kimDynamicNonmyopicPortfolio1996} showed that in some extreme cases, investors could achieve unbounded utility in a finite time. However, they also noted that this situation does not arise in practice.    
\end{remark}

We also introduce the exponentially-transformed criterion $I(h;T,\theta,r_0) := e^{-\theta J(h;T,\theta, r_0)}$ such that:
\begin{align}\label{eq:Criterion:I}
    I(h;T,\theta,r_0) = r_0^{-\theta} \mathbf{E} \left[e^{-\theta R_T}\right]=r_0^{-\theta}\mathbf{E} \left[    
 \exp \left\{ \theta \int_{0}^{T} g(t,X_t, Y_t,h_t;\theta) dt \right\} \chi_T^h\right]    
\end{align}
Hence, for $\theta >0$, maximizing $J(h;T,\theta,r_0)$ is equivalent to minimizing $I(h;T,\theta,r_0)$. The maximization and minimization are reversed for $\theta <0$.

\begin{remark}
The risk-sensitive criteria $J$ and $I$ are consistent with utility maximization. Specifically, rescaling $I$ by the constant $-\frac{r_0^{\theta}}{\theta}$ yields $-\frac{r_0^{\theta}}{\theta} I(h;T,\theta,r_0) = -\frac{1}{\theta}\mathbf{E} \left[ \left(\frac{V_T}{L_T}\right)^{-\theta}\right]$, which is the power utility function with risk aversion parameter $\vartheta = - \theta$, for the investor's wealth relative to the benchmark. This observation connects our research with the rich literature on utility maximization, which started with \citet{Merton1969}. \citet[Section 2.2 in][]{DavisLleoBook2014} discusses this and other advantageous properties of the risk-sensitive criterion.
\end{remark}

Following \citet{na00}, we want $\chi_t^h$ in \eqref{eq:expmartX} to be an exponential martingale so that it can be used to define a new measure $\mathbb{P}^h$ via the Radon-Nikodym derivative
\begin{align}\label{def:Ph}
    \frac{d\mathbb{P}^h}{d\mathbb{P}} \Big{|}_{\mathcal{F}_t}
    = \chi_t^h, \quad t \in [0,T],
\end{align}
and an associated standard $\left(\mathcal{F}_t,\mathbb{P}^h\right)$-Wiener process
\begin{align}\label{eq:Ph:BM:Wh:X}
    W^{h}_t 
    := W_t 
    + \theta \int_{0}^t \left(\Sigma^{(1)'}\left(s, Y_s\right) h_s - \Xi'\left(s, Y_s\right) \right) ds.
\end{align}
For this purpose, in this subsection, we consider as admissible those strategies that correspond to the following
\begin{definition}[Class $\mathcal{A}_X (T)$]\label{def:control:class:AX}
A $\mathbb{R}^{m_1}$-valued control process $h_t$ is in class $\mathcal{A}_X(T)$ if the following conditions are satisfied:
\begin{enumerate}[i)]
    \item $h_t$ is progressively measurable with respect to $\left\{\mathcal{B}([0,t]) \otimes \mathcal{F}^Y_t\right\}_{t \geq 0}$ and is c\`adl\`ag;

    \item $P\left(\int_{0}^{T} \lvert h_s \rvert^2 ds < +\infty \right) = 1$;
    
    \item The following Kazamaki conditions \citep{kazamakiEquivalenceTwoConditions1977} hold: 
    \begin{align}
        \mathbf{E} \left[ \exp \left\{ -\frac{\theta}{2} \int_{0}^{t}\left(h_s' \Sigma^{(1)}\left(s, Y_s\right) - \Xi\left(s, Y_s\right) \right)dW_s\right\} \right] < \infty;
                                        \label{def:control:class:AX:Kaza1}\\
        \mathbf{E}^h \left[ \exp \left\{ 
        - \frac{1}{2} \int_0^t a^{\dagger}(s, X_s,Y_s; h_s)' \left( \Sigma^Y\Sigma^{Y'}(s,Y_s)\right)^{-1}\Sigma^Y\left(s, Y_s\right) dW^h_s \right\}\right] < \infty,     
                                        \label{def:control:class:AX:Kaza2}
    \end{align}
    where $\mathbf{E}^h[\cdot]$ denotes the expectation with respect to the measure $\mathbb{P}^h$ and 
\begin{align}\label{eq:a_dag}
a^\dagger(t, X_t, Y_t; h_t) 
&:= a^Y(t, X_t, Y_t) 
    - \theta \Sigma^Y\left(t, Y_t\right)\left(\Sigma^{(1)'}\left(t, Y_t\right) h_t - \Xi'\left(t, Y_t\right) \right),
\end{align}
        
\end{enumerate}    
\end{definition}

\begin{remark}\label{R.Kazamaki_vs_Novikov}
Like the well-known Novikov condition, the Kazamaki condition provides a sufficient condition for Dol\'eans exponentials to be exponential martingales. The main difference between the two is that while Novikov introduces an assumption on the quadratic variation, Kazamaki sets a slightly stronger condition directly on the Brownian motion  -- see II.5 in \citet{ikwa81} or III.8 in \citet{Pr05}. Hence, the Kazamaki condition is better suited for our investigation into the relation between Brownian motions under various changes of measure.
\end{remark}

\begin{remark}\label{R.4}
    Under the Kazamaki condition \eqref{def:control:class:AX:Kaza1}, the Dol\'eans-Dade exponential $\chi_t^h$ defined at \eqref{eq:expmartX} is an exponential martingale on $(\mathbb{P},\mathcal{F}_t)$ for $t \in [0,T]$. The Kazamaki condition \eqref{def:control:class:AX:Kaza2} guarantees that we can define a new measure $\bar{P}$ for which the observation process is a martingale. This is an essential ingredient in nonlinear filtering and a cornerstone of Nagai and Peng's approach.
\end{remark}

The dynamics of $(X_t, Y_t)$ becomes, always under $\mathbb{P}^h$,
\begin{numcases}{}
dX_t = \left[b(t,X_t,Y_t)dt -\theta \Lambda(t,X_t,Y_t) \left(\Sigma^{(1)'}(t, Y_t) h_t - \Xi'(t, Y_t) \right)\right]dt + \Lambda(t,X_t,Y_t) dW^h_t \label{eq:state:X:Ph}\\
dY_t= \left[a^Y(t,X_t,Y_t)-\theta \Sigma^Y(t,Y_t) \left(\Sigma^{(1)'}(t, Y_t) h_t - \Xi'(t, Y_t)\right) \right] dt+ \Sigma^Y(t,Y_t)dW^h_t\label{eq:obs:Y:Ph} 
\end{numcases}
with $a^Y(t,X_t,Y_t)$ according to \eqref{eq:Y:coef:decomp:a} and the risk-sensitive criteria are
\begin{align}\label{eq:Jh}
    J^h(h;T,\theta,r_0) = -\frac{1}{\theta} \ln \mathbf{E}^h
        \left[ r_0^{-\theta} \exp \left\{ \theta \int_{0}^{T} g(t,X_t,Y_t,h_t;\theta) dt \right\} \right]
\end{align}
Analogously to criterion $I$ at \eqref{eq:Criterion:I}, here as well, we can define an exponentially transformed criterion $I^h$ in connection with $J^h$.

\subsection{The separated control problem for a finitely parametrized filter}\label{S.4.2}

We come now to describe the problem in terms of the filter, parametrized by the vector $\zeta$ according to Assumption \ref{as:FDfilter}. Recalling the definition of $\hat a^Y(t,Z_t)$ at \eqref{eq:ahat} and the ensuing decomposition $\hat{a}(t,Z_t) = \begin{pmatrix} \hat{a}^{(1)'}(t,Z_t) & \hat{a}^{(2)'}(t,Z_t) \end{pmatrix}'$ as well as Lemma \ref{lem:Wiener:restricted}, and paralleling the description in Subsection \ref{S.4.1}, we introduce now what we call the \textit{ separated problem} with dynamics for $Z_t=(Y_t,\zeta_t)'$ given by \eqref{eq:obs:tilde} and \eqref{eq:zeta:wtilde}. 

Recall also that the risk-sensitive criterion was introduced in \eqref{eq:J} in terms of the excess return $R_t$ that satisfies \eqref{eq:R} in the full filtration ${\mathcal{F}_t}$. We shall next consider it in the observation subfiltration ${\mathcal{F}_t^Y}$. For this purpose we shall start from the dynamics of $V_t$ and $R_t$ given at
\eqref{eq:V} and \eqref{eq:R} respectively and express them in terms of ${\mathcal{F}_t^Y}$-adapted quantities. The tool to this effect is given by Lemma \ref{lem:Wiener:restricted} that concerns the entire observation vector $Y_t$ and that extends beyond the market assets. This has to be taken into account when considering the investment strategy in view of the dynamics of $V_t$ as it concerns only the assets in the market. 
For this purpose, we shall introduce the following
\begin{notation}\label{N.2}
Given an $m_1-$dimensional investment strategy $h_t$ as in Section \ref{S.4.1}, let $\tilde{h}_t $ be the $m^Y$-element vector process defined as
    $
    \tilde{h}_t  =
    \begin{pmatrix}
        0_{\ell}'    &
        h_t'         &
        0_{m_2 + k +1}'
    \end{pmatrix}'$
    where $0_{p}$ denotes the zero column vector with $p$ elements for some $p \in \mathbb{N}$, and note that  $\tilde{h}_t'\Sigma^{Y}_t  = h_t \Sigma^{(1)}_t $.
    \end{notation}
    
 Starting from  \eqref{eq:V} and using  Lemma \ref{lem:Wiener:restricted} we then obtain for the same $V_t$, but in the filtration  ${\mathcal{F}_t^Y}$, the following  
 \begin{align}\label{eq:Vtilde}
dV_t
&= V_t h_t' a^{(1)}\left(t,X_t, Y_t\right)dt
+ V_t h_t' \Sigma^{(1)} \left(t, Y_t\right)	 dW_t
                                    \nonumber\\
&= V_t h_t' a^{(1)}\left(t,X_t, Y_t\right)dt
+ V_t \tilde{h}_t' \left[ \left(\hat{a}^Y(t,Z_t) - a^Y(t,X_t,Y_t) \right)dt + \Sigma^Y\left(t, Y_t\right) d\tilde{W}_t \right]
                                \nonumber\\
&= V_t h_t' a^{(1)}\left(t,X_t, Y_t\right)dt\nonumber\\
&\>+ V_t   h_t' \left[ \left(\hat{a}^{(1)}(t,Z_t) - \frac{1}{2} d_{\Sigma^{(1)}}(t,Y_t) \right) 
- \left( a^{(1)}(t,X_t,Y_t)- \frac{1}{2} d_{\Sigma^{(1)}}(t,Y_t) \right) \right] dt
                                + V_t h_t' \Sigma^{(1)}\left(t, Y_t\right) d\tilde{W}_t
                                \nonumber\\
&= V_t h_t' \hat{a}^{(1)}(t,Z_t)dt + V_t h_t' \Sigma^{(1)}\left(t, Y_t\right) d\tilde{W}_t,
\end{align}
In full analogy, for $R_t$ in \eqref{eq:R}, we obtain in the filtration ${\mathcal{F}_t^Y},$
\begin{align}\label{eq:Rtilde}
dR_t &= \left[ 
  - \frac{1}{2} h_t'\Sigma^{(1)}\Sigma^{(1)'} \left(t, Y_t\right) h_t + h_t'\hat{a}^{(1)}(t,Z_t) + \frac{1}{2}\Xi\Xi'\left(t, Y_t\right) - \hat{c}(t,Z_t) \right] dt
                                    \nonumber\\
    & + \left(h_t'\Sigma^{(1)}\left(t, Y_t\right) - \Xi\left(t, Y_t\right) \right) d\tilde{W}_t,
\qquad
R_0 = \ln \frac{v}{l} = r_0.     
\end{align}

The risk-sensitive benchmarked criterion is now
\begin{align}\label{eq:Jhat}
\hat J(h;T,\theta,r_0) 
&:= -\frac{1}{\theta} \ln \mathbf{E} \left[ r_0^{-\theta} e^{-\theta R_T}\right]
= -\frac{1}{\theta} \ln \mathbf{E} \left[    
r_0^{-\theta} \exp \left\{ \theta \int_{0}^{T} \hat g(t,Z_t, h_t;\theta) dt \right\} \hat{\chi}_T^h
\right],
\end{align}
where, again, $\theta \in (-1,0)\cup(0,\infty)$ is the risk sensitivity parameter, $T < \infty$ is a fixed time horizon,
\begin{align}\label{eq:gh}
\hat{g}(t,z,h;\theta)
&:= \frac{\theta+1}{2} h'\Sigma^{(1)}\Sigma^{(1)'} \left(t, y\right) h 
- h' \hat{a}^{(1)}(t,z)
- \theta h'\Sigma^{(1)}\Xi'\left(t,y\right)
+ \hat{c}(t,z)
                                        \nonumber\\
&+ \frac{\theta -1}{2}\Xi\Xi'\left(t,y\right),
\end{align}
the vector $y$ contains the first $m^Y$ components of the vector $z$, and the Dol\'eans-Dade exponential $\hat{\chi}_t^h, \; t \in [0,T]$ is defined via  
\begin{align}\label{eq:expmarthat}
    \hat{\chi}_t^h 
    &:= \exp \bigg\{ -\theta \int_{0}^{t}\left(h_s' \Sigma^{(1)}\left(s, Y_s\right) - \Xi\left(s, Y_s\right) \right)d \tilde{W}_s  
                    \nonumber\\  
    &
        -\frac{1}{2} \theta^2     \int_{0}^{t} \left(h_s' \Sigma^{(1)}\left(s, Y_s\right) - \Xi\left(s, Y_s\right) \right)
        \left(\Sigma^{(1)'}\left(s, Y_s\right) h_s - \Xi'\left(s, Y_s\right) \right)ds 
    \bigg\}.
\end{align}
The exponentially-transformed criterion is here $\hat I(h;T,\theta,r_0) = e^{-\theta \hat J(h;T,\theta, r_0)}$.

As in the previous Subsection \ref{S.4.1}, we want $\hat{\chi}_t^h$ in \eqref{eq:expmarthat} to be an exponential martingale so that it can be used to define a new measure $\hat{\mathbb{P}}^h$ via the Radon-Nikodym derivative
\begin{align}\label{def:Phhat}
    \frac{d\hat{\mathbb{P}}^h}{d\mathbb{P}} \Big{|}_{\mathcal{F}^Y_t} = \hat{\chi}_t^h, \quad t \in [0,T],
\end{align}
with associated standard  $(\hat{\mathbb{P}}^h,{\mathcal{F}_t^Y)}$-Wiener process
\begin{align}\label{eq:Ph:BM:What}
    \tilde{W}^{h}_t
    := \tilde{W}_t 
    + \theta \int_{0}^t \left(\Sigma^{(1)'}\left(s, Y_s\right) h_s - \Xi'\left(s, Y_s\right) \right) ds.
\end{align}
For this purpose, this time we consider as admissible strategies those in a class that we denote by $\mathcal{A}_Z(T)$ and that is given by the following

\begin{definition}[Class $\mathcal{A}_Z(T)$]\label{def:control:class:AZ}
A $\mathbb{R}^{m_1}$-valued control process $h_t$ is in class $\mathcal{A}_Z(T)$ if the following conditions are satisfied:
\begin{enumerate}[i)]
    \item $h_t$ is progressively measurable with respect to $\left\{\mathcal{B}([0,t]) \otimes \mathcal{F}^Y_t\right\}_{t \geq 0}$ and is c\`adl\`ag;

    \item $P\left(\int_{0}^{T} \lvert h_s \rvert^2 ds < +\infty \right) = 1$;

    \item The following Kazamaki conditions hold: 
    \begin{align}
        \mathbf{E} \left[ \exp \left\{ -\frac{\theta}{2} \int_{0}^{t}\left(h_s' \Sigma^{(1)}\left(s, Y_s\right) - \Xi\left(s, Y_s\right) \right)d \tilde{W}_s\right\} \right] < \infty;
                                        \label{def:control:class:AZ:Kaza1}\\
        \hat{\mathbf{E}}^h \left[ \exp \left\{ 
        - \frac{1}{2} \int_0^t \check{a}^{Y'}(s, Z_s; h_s) \left( \Sigma^Y\Sigma^{Y'}(s,Y_s)\right)^{-1}\Sigma^Y\left(s, Y_s\right) d\tilde{W}^h_s \right\}\right] < \infty,
                                        \label{def:control:class:AZ:Kaza2}
    \end{align}
    where $\hat{\mathbf{E}}^h[\cdot]$ denotes the expectation with respect to the measure $\hat{\mathbb{P}}^h$  and
    \begin{align}\label{eq:a_check}
\check{a}^Y(t, Z_t; h_t) 
&:= \hat{a}^Y(t, Z_t) 
    - \theta \Sigma^Y\left(t, Y_t\right)\left(\Sigma^{(1)'}\left(t, Y_t\right) h_t - \Xi'\left(t, Y_t\right) \right).
\end{align}
\end{enumerate}    
\end{definition}

\begin{remark}\label{R.5}
    As already seen in Remark \ref{R.4}, the Kazamaki conditions give a sufficient condition for the Dol\'eans-Dade exponential to be an exponential martingale. Under condition \eqref{def:control:class:AZ:Kaza1}, $\hat{\chi}_t^h$ defined at \eqref{eq:expmarthat} is an exponential martingale on $\left(\mathbb{P}, \mathcal{F}^Y_t\right)$ for $t \in [0,T]$. Under condition \eqref{def:control:class:AZ:Kaza2}, $\bar{\chi}^Z_t$ defined at \eqref{def:PbarZ} below is an exponential martingale on $\left(\hat{\mathbb{P}}^h, \mathcal{F}^Y_t\right)$ for $t \in [0,T]$.
\end{remark} 

\begin{remark}
      The class $\mathcal{A}_Z(T)$ is appropriate for existing risk-sensitive investment management models. Conditions i) and ii) are standard. The first condition in iii) comes from Kuroda and Nagai's solution technique, which articulates around a change of measures \citep{kuna02}. The optimal control satisfies this condition in diffusion models \citep{kuna02,dall_RSBench,dall_BLcontinuous, dall_BBL_2020, davisRisksensitiveBenchmarkedAsset2021}. As long as we have diffusion-type models, the control is a function of $X_t$, a diffusion, so the first condition in iii) holds under standard assumptions on the coefficients. The second condition in iii) is required at \eqref{def:PbarZ} to perform a change of measure and derive the Modified Zakai Equation below in Section \ref{sec:Zakai}. The only difference with the first part of condition iii) is that $a$ plays a role via the estimate $\hat{a}^Y$.
\end{remark}

The $\hat{\mathbb{P}}^h$-dynamics of the filter parameters $\zeta_t$ and the observations $Y_t$ are then
\begin{numcases}{}
    d \zeta_t = \left[\hat{G}(t, Z_t)- \theta H(t, Z_t)\Sigma^Y\left(t, Y_t\right) \left(\Sigma^{(1)'}\left(t, Y_t\right) h_t - \Xi'\left(t, Y_t\right) \right) \right] dt
                            \nonumber\\
    \phantom{d \zeta_t = } + H(t, Z_t)\Sigma^Y\left(t, Y_t\right) d\tilde{W}^h_t, 
                    \label{eq:zeta:wh}\\
    dY_t = \check{a}^Y(t, Z_t; h_t) dt 
    + \Sigma^Y\left(t, Y_t\right) d\tilde{W}^h_t,
                    \label{eq:obs:wh}
\end{numcases}
and  $\check{a}^Y(t, Z_t; h_t) $  is as defined at \eqref{eq:a_check}
Moreover, the risk-sensitive control criterion becomes
\begin{align}
    \hat J^h(h;T,\theta,r_0) 
    := -\frac{1}{\theta} \ln \hat{\mathbf{E}}^h
        \left[ r_0^{-\theta} \exp \left\{ \theta \int_{0}^{T} \hat{g}(t,Z_t,h_t;\theta) dt \right\} \right]   
    =: - \frac{1}{\theta} \ln  \hat I^h(h;T,\theta,r_0) 
                                    \label{eq:Jhh}
\end{align}

\subsection{Equivalence of the original and the separated problems}\label{S.4.3}

The criterion \eqref{eq:Jhh} for the separated problem was obtained directly from the corresponding criterion \eqref{eq:Jh} in the original problem 
using the crucial Lemma \ref{lem:Wiener:restricted}. Nevertheless, to obtain the full equivalence of the two problem formulations under the information structure given by the filtration 
 $\mathcal{F}_t^Y$, we still have to show first that the measure ${\mathbb{P}}^h$ defined at \eqref{def:Ph}  is identical to  $\hat{\mathbb{P}}^h$   defined at \eqref{def:Phhat} and that, since $h_t$ has to be $\mathcal{F}_t^Y$-adapted in both cases, the admissible class of strategies  $\mathcal{A}_X (T)$ in the original problem (see Definition \ref{def:control:class:AX} ) coincides with $\mathcal{A}_Z (T)$ given in Definition \ref{def:control:class:AZ}. To this effect, we show the following

\begin{proposition}\label{prop:P_h:unique}
\hspace{0cm}

\begin{enumerate}[i)]
    \item The Kazamaki conditions \eqref{def:control:class:AX:Kaza1} , \eqref{def:control:class:AX:Kaza2} in Definition \ref{def:control:class:AX} and \eqref{def:control:class:AZ:Kaza1}, \eqref{def:control:class:AZ:Kaza2} in Definition \ref{def:control:class:AZ} are equivalent.    
    \item The relation between the Wiener process $W^h_t$ on $(\mathbb{P}^h, \mathcal{F}_t)$ and the Wiener process $ \tilde{W}^h_t $ on $(\hat{\mathbb{P}}^h, \mathcal{F}^Y_t)$ is
    \begin{align}\label{eq:Wh_to_tildeWh}
        \Sigma^Y\left(t, Y_t\right) d \tilde{W}_t^h 
        = \left(a^Y(t,X_t,Y_t) - \hat{a}^Y(t,Z_t) \right)dt + \Sigma^Y\left(t, Y_t\right) dW^h_t.
    \end{align}
    and the measures $\mathbb{P}^h$ defined via \eqref{def:Ph} and $\hat{\mathbb{P}}^h$ defined via \eqref{def:Phhat} are identical, that is, $\hat{\mathbb{P}}^h=\mathbb{P}^h$.
    \end{enumerate}
    
\end{proposition}

\begin{proof}
See Appendix \ref{app:proofs:controlpb:equival}.
\end{proof}

\begin{remark}
Although the measure $\mathbb{P}^h$ is the same, the filtrations $\mathcal{F}_t$ for $(W_t, W^h_t)$ and $\mathcal{F}^Y_t$ for $(\tilde{W}_t,  \tilde{W}^h_t )$ differ.
\end{remark}

As an immediate corollary to this Proposition  \ref{prop:P_h:unique}  we have
\begin{corollary}\label{equivad}
    The admissible classes $\mathcal{A}_X(T)$ in Definition \ref{def:control:class:AX} and $\mathcal{A}_Z (T)$ in Definition \ref{def:control:class:AZ} are strictly equivalent.
\end{corollary}

With the help of Proposition \ref{prop:P_h:unique} and Corollary \ref{equivad} we can now obtain the main result of this subsection regarding the equivalence of the original and the separated control problems when the information structure is given by the filtration  $\mathcal{F}^Y_t$, namely
\begin{theorem}\label{tequiv}
    Under Assumption \ref{as:FDfilter}, for an admissible control $h_t$ satisfying conditions (i), (ii), and \eqref{def:control:class:AX:Kaza1} in Definition \ref{def:control:class:AX} (or equivalently conditions (i), (ii), and \eqref{def:control:class:AZ:Kaza1} in Definition \ref{def:control:class:AZ}), we have $
        J^h(h;T,\theta,r_0) = \hat{J}^h(h;T,\theta,r_0) 
    $ and therefore also
    $
        I^h(h;T,\theta,r_0) = \hat{I}^h(h;T,\theta,r_0). 
    $
\end{theorem}

Finally, as a further corollary to Proposition \ref{prop:P_h:unique} we obtain also the following property for the change of Brownian motion, namely
\begin{corollary}
The translations of the Brownian motion induced by Lemma 3.2 and by the change of measure from $\mathbb{P}$ to $\mathbb{P}^h$ are commutative. Said otherwise, it does not matter whether we apply Lemma 3.2 before or after performing the change of measure; we will get to the same Brownian motion $\tilde{W}_t^h$. 
\end{corollary}

\section{Modified Zakai Equation}\label{sec:Zakai}

The so-called modified Zakai equation (MZE) was introduced in \citet{na00} and \citet{nape02} as a useful tool in view of the solution of stochastic control problems under partial observation, in particular for the solution of risk-sensitive problems of the type considered in this paper. Appendix \ref{app:MZE:NagaiPeng} summarizes this approach and its main results. In this section, we derive an MZE in the more general setting where, according to Assumption \ref{as:FDfilter} a finite-dimensional filter process $\zeta_t$ exists. We adapt the approach by Nagai  
and Peng to this setting, more precisely to the separated problem with  
a finite-dimensional filter as formulated in Subsection \ref{S.4.2}. A major difference exists between our setting and that of Nagai and Peng. In Nagai and Peng's case, the filter enters only at the level of the solution of the MZE, whereas in our case, the filter intervenes right from the beginning in the formulation of the problem in its separated form. This difference implies that the approach of Nagai and Peng results in a stochastic partial differential equation for a certain operator. This SPDE is intricate to formalize and challenging to solve in general. Nevertheless, Nagai and Peng show that in the linear-Gaussian case, one obtains an explicit analytic solution that exploits the specificity of the Kalman Filter. As a byproduct, they also show that a separation of estimation and control holds in this case. In fact, one could adapt their result to the more general case of nonlinear, but Gaussian, dynamics where the filter is the Kalman filter applied after linearizing the nonlinear coefficients (\textit{Extended Kalman Filter}).

The fact that we let the filter intervene right from the beginning in the formulation of the problem in its separated form leads to a deterministic PDE given in Theorem \ref{theo:Zakai:Z} below that concerns the MZE for the separated problem relating to a corresponding operator $q_Z^h(t)(\varphi_t)$ given in Definition  \ref{def:qZ} below. Also this equation is not easy to formalize in a standard way but, by its very derivation,  it is satisfied by the operator $q_Z^h(t)(\varphi_t)$ (see Definition  \ref{def:qZ} below) that, for the test function $\varphi_t=1$, equals the value function of the separated problem which was derived directly via our approach without passing through the MZE. This allows us to investigate better when and in what form a separation of estimation and control holds by using either an exact or an approximate finite-dimensional filter (see Section
\ref{sec:Separability} and the examples in Section \ref{sec:examples}).

Before coming to derive the equation corresponding to the MZE approach for the separated problem,
we also state and prove a result that extends the existing Kallianpur-Striebel formula to our setting. This generalized Kallianpur-Striebel formula relates the risk-sensitive criterion $I^h$ under the control measure $\mathbb{P}^h$ with that evaluated under the filtering measure $\bar{\mathbb{P}}$:

\begin{proposition}[Generalized Kallianpur-Striebel Formula]\label{prop:GKSF}
Assume $\varphi \in C_b^{1,2}$, then
    \begin{align}\label{eq:GKSF}
       \mathbf{E}^h \left[ \varphi(t,X_t;Y_t) e^{\theta \int_0^t  g(s, X_s, Y_s,h_s;\theta) ds} \mid \mathcal{F}^Y_t \right]
       &= \frac{\bar{\mathbf{E}} \left[ \varphi(t,
       X_t; Y_t) e^{\theta \int_0^t g(s, X_s, Y_s, h_s;\theta) ds} \Psi^X_t \mid \mathcal{F}^Y_t \right]}{\bar{\mathbf{E}} \left[\Psi^X_t \mid \mathcal{F}^Y_t \right]}
                                                               \nonumber\\
       &= \frac{q_X^h(t)(\varphi_t)}{\bar{\mathbf{E}} \left[ \Psi^X_t \mid \mathcal{F}^Y_t \right]},
\end{align}
where $\mathbf{E}^h \left[ \cdot \right]$ is the expectation with respect to the measure $\mathbb{P}^h$ and $\bar{\mathbf{E}} \left[ \cdot \right]$ is the expectation with respect to the measure $\bar{\mathbb{P}}$.
\end{proposition}

\begin{proof} [Proof of Proposition  \ref{prop:GKSF}]
See Appendix \ref{app:proof:GKSF}.
\end{proof}


We come now to the problem formulation as described in Section \ref{S.4.2}, namely the separated problem formulated in terms of the filter that is characterized by a finite-dimensional parameter vector $\zeta_t$.

Analogously to \citet{na00}, it is convenient to have $Y_t$ a martingale by passing to a new measure. We denote this measure by $\bar{\mathbb{P}}$.
Assuming that $h_t\in  \mathcal{A}_Z(T)$ (see Definition \ref {def:control:class:AZ}) let
\begin{multline}\label{def:PbarZ}
       \bar{\chi}^Z_t
       := \exp \left\{
           - \int_0^t \check{a}^{Y}(s, Z_s; h_s)' \left( \Sigma^Y\Sigma^{Y'}(s,Y_s)\right)^{-1}\Sigma^Y\left(s, Y_s\right) d\tilde{W}^h_s
                           \right. \\
          \left.
          - \frac{1}{2}\int_0^t \check{a}^{Y}(s, Z_s; h_s)' \left( \Sigma^Y \Sigma^{Y'}\left(s, Y_s\right)\right)^{-1}\check{a}^Y(s,Z_s; h_s) ds \right\}
\end{multline}
for $t \in [0,T]$ and with $\check{a}^Y(t, Z_t; h_t) $  as in \eqref{eq:a_check}. It is a martingale on $(\bar{\mathbb{P}}, \mathcal{F}^Y_t)$ corresponding to the Radon-Nikodym derivative $\frac{d\bar{\mathbb{P}}}{d\hat{\mathbb{P}}^h} \Big{|}_{\mathcal{F}^Y_t}.$  The process
\begin{align}\label{eq:Wiener:PbarZ}
       \check{W}_t
       :=\tilde{W}^{h}_t
       + \int_{0}^t \Sigma^{Y'}\left(s, Y_s\right)\left( \Sigma^Y\Sigma^{Y'}\left(s, Y_s\right)\right)^{-1} \check{a}^Y(s, Z_s; h_s) ds
\end{align}
is then a standard $(\bar{\mathbb{P}}, \mathcal{F}^Y_t)-$ Wiener process  and the $\bar{\mathbb{P}}$-dynamics of the filter parameter $\zeta_t$ and of the observation $Y_t$ are
\begin{numcases}{}
       d \zeta_t = G(t, Z_t)dt + H(t, Z_t)dY_t,
       \label{eq:zeta:wbar}\\
       dY_t = \Sigma^Y\left(t, Y_t\right) d\check{W}_t.
       \label{eq:Y1:Pbar}
\end{numcases}
The corresponding risk-sensitive criterion $\bar{J}(h;T,\theta,r_0)$ and exponentially-transformed criterion $\bar{I}(h;T,\theta,r_0)$ are
\begin{numcases}{}
       \bar{J}(h;T,\theta,r_0) =  -\frac{1}{\theta} \ln \bar{\mathbf{E}} \left[ r_0^{-\theta} \exp \left\{ \theta \int_0^T \hat{g}(t, Z_t, h_t;\theta) dt\right\} \Psi^Z_T\right].
       \label{eq:Jbar}\\
       \bar{I}(h;T,\theta,r_0) = r_0^{-\theta} \bar{\mathbf{E}} \left[\exp \left\{ \theta \int_0^T \hat{g}(t, Z_t, h_t;\theta) dt\right\} \Psi^Z_T\right].
       \label{eq:Ibar}
\end{numcases}
where we have put  $\Psi^Z_t:= \left(\bar{\chi}^Z_t\right)^{-1}$ and $\hat g$ is as defined at \eqref{eq:gh}.

Inspired by \citet{na00}, we introduce the following

\begin{definition}\label{def:qZ}
For $h_t \in \mathcal{A}_Z(T)$ and $\varphi(t,\zeta;y): [0,T] \times \mathbb{R}^{q} \times \mathbb{R}^{m^Y}  \to \mathbb{R}$ a $C_b^{1,2}$ test function indexed by $y$, set
\begin{align}\label{eq:qZ}
       q_Z^h(t)(\varphi_t) := \bar{\mathbf{E}} \left[e^{ \theta \int_0^t \hat{g}(s, Z_s, h_s; \theta)ds} \Psi^Z_t \varphi(t, \zeta_t; Y_t)  \right],
\end{align}
so that
$
\bar{I}(h;T,\theta,r_0)
       = r_0^{-\theta}  \left[ q_Z^h(T)(1) \right].
$
\end{definition}

\begin{remark}\label{R-operators}
Contrary to \citet{na00}, here  the operator $q_Z^h(t)(\varphi_t)$ is deterministic and has therefore to satisfy a deterministic PDE which will be given in the next Theorem \ref{theo:Zakai:Z} and which concerns the MZE
for the separated problem.
\end{remark}

In view of this theorem let  $\hat L^h$ be the generator of the process $\zeta_t$ under the measure $\hat{\mathbb{P}}^h$, with $y$ treated as a fixed parameter. We have now

\begin{theorem}\label{theo:Zakai:Z}
The operator $q_Z^h(t)(\varphi_t)$ from Definition  \ref{def:qZ} satisfies the following PDE
\begin{align}\label{eq:Zakai:PDE:Z:SPDE}
&q_Z^h(t)(\varphi_t) - q_Z^h(0)(\varphi_0)
                    \nonumber\\
&= \int_0^t q_Z^h(s)\left( 
        \frac{\partial \varphi}{\partial t}(s,\zeta_s;Y_s) 
        + \hat L^h\varphi(s,\zeta_s,Y_s) 
        + \theta \hat g(s,Z_s,h_s;\theta) \varphi(s,\zeta_s;Y_s)\right) ds.
\end{align}  
where
\begin{align}\label{eq:operator:Lh:Z}
    \hat L^h\varphi(t,\zeta;y) 
    :=& \frac{1}{2} \tr \left\{ H(t, \zeta;y)\Sigma^Y\left(t; y\right)\Sigma^{Y'}\left(t; y\right) H'(t, \zeta;y) D^2\varphi \right\}
                    \nonumber\\
    +& \left[\hat{G}(t, \zeta;y)- \theta H(t, \zeta;y)\Sigma^Y\left(t;y\right) \left(\Sigma^{(1)'}\left(t;y\right) h - \Xi'\left(t;y\right) \right) \right] D\varphi
                    \nonumber\\
    =& \frac{1}{2} \sum_{i,j=1}^q \left\{ H(t, \zeta;y)\Sigma^Y\left(t; y\right)\Sigma^{Y'}\left(t; y\right) H'(t, \zeta;y)\right\}_{ij} \frac{\partial^2 \varphi}{\partial \zeta_i \partial \zeta_j}
                    \nonumber\\
    +& \sum_{i=1}^q \left\{\hat{G}(t, \zeta;y)- \theta H(t, \zeta;y)\Sigma^Y\left(t;y\right) \left(\Sigma^{(1)'}\left(t;y\right) h - \Xi'\left(t;y\right) \right) \right\}_i \frac{\partial \varphi}{\partial \zeta_i}
\end{align}
is the $\hat{\mathbb{P}}^h$-generator related to the process $\zeta_t$, $D\varphi = \begin{pmatrix} \frac{\partial \varphi}{\partial \zeta_1} & \ldots & \frac{\partial \varphi}{\partial \zeta_i} & \ldots & \frac{\partial \varphi}{\partial \zeta_q} \end{pmatrix}$ and $D^2\varphi = \left[ \frac{\partial^2 \varphi}{\partial \zeta_i \zeta_j} \right], i,j = 1, \ldots, q$.
\end{theorem}

\begin{proof}[Proof of Theorem \ref{theo:Zakai:Z}.]
See Appendix \ref{app:proof:Zakai:Z}.
\end{proof}

\begin{remark}
Note that although $q_Z^h(t)(\varphi_t)$ is defined under the measure $\bar{\mathbb{P}}$, the drift in \eqref{eq:Zakai:PDE:Z:SPDE} depends on  $\hat L^h$. How did we get there? The derivation starts with the $\bar{\mathbb{P}}$-operator of $\zeta_t$. This operator then acquires an extra term through the cross variation $d \langle \Psi^Z, \varphi \rangle_t$, yielding the $\hat{\mathbb{P}}^h$-operator $\hat L^h$.
\end{remark}

\begin{remark}
We can express the conditional probability $q_Z^h(t)(\varphi_t)$ in terms of its conditional density $q_Z^h(t,\zeta;y)$ as
\begin{align}
    q_Z^h(t)\left(\varphi(t)\right) = \int q_Z^h(t,\zeta;y) e^{ \theta \int_0^t g(s,y,\zeta,h;\theta)ds} \varphi(t,\zeta;y) d\zeta,
\end{align}
where the density $q_Z^h(t,\zeta;y)$ is with respect to the Lebesgue measure. By standard PDE argument \cite[see for instance Chapter 1 in][]{fr64}, we can show that the conditional density $q_Z^h(t)$ satisfies the parabolic PDE:
\begin{align}\label{eq:Zakai:PDE:Z:PDE}
    \frac{\partial q_Z^h(t)}{\partial t} - \hat L^{h*}q_Z^h(t) = 0
\end{align}
where the operator $\hat L^{h*}$ is the adjoint of the operator $\hat L^{h}$ and is thus given by
\begin{align}\label{eq:operator:Lh:Z:adjoint}
    \hat L^{h*}f(t,\zeta;y) 
    :=& \frac{1}{2} \sum_{i,j=1}^q \frac{\partial^2 }{\partial \zeta_i \partial \zeta_j} \left( \left\{H(t, \zeta;y)\Sigma^Y\left(t; y\right)\Sigma^{Y'}\left(t; y\right) H'(t, \zeta;y)\right\}_{ij} f(t,\zeta;y) \right)
    \nonumber\\
    -& \sum_{i=1}^q \frac{\partial }{\partial \zeta_i} \left( \left\{ \hat{G}(t, \zeta;y)- \theta H(t, \zeta;y)\Sigma^Y\left(t;y\right) \left(\Sigma^{(1)'}\left(t;y\right) h - \Xi'\left(t;y\right) \right) \right\}_i f(t,\zeta;y) \right)
\end{align}
This parabolic PDE is the forward Kolmogorov or Fokker-Planck equation for the filter parametrized by $\zeta_t$ under the measure $\hat{\mathbb{P}}^h$. The existence and uniqueness of a solution are tied to the properties of the drift and diffusion functions $G(t, Y_t, \zeta_t)$, $H(t, Y_t, \zeta_t)$, $\hat{a}^Y(t, Y_t, \zeta_t)$, and $\Sigma^Y\left(t,Y_t\right)$. When it exists, this solution is the filter process's density.
\end{remark}

\section{ Criteria for Separability}\label{sec:Separability}

In this section, we shall derive a criterion according to which one can conclude whether and to what extent a separation property holds for a given risk-sensitive control problem under partial observations, as we consider it in this paper concerning an investment problem. We recall that we consider a separation property to hold if the separated problem, namely the problem where instead of the unobserved factor process $X_t$ one considers its filter distribution, can be expressed in terms of the filter parameters or, equivalently, by one or more synthetic values of the filter distribution. 

For this purpose, recalling the definition at \eqref{eq:ahat} and making explicit $Z_t=(Y_t,\zeta_t)$, we have
\begin{align}\label{a:h}
\hat a(t,Y_t,\zeta_t)=\int a(t,x,Y_t)\,dp(x,\zeta_t)\>\>\>,\quad \hat c(t,Y_t,\zeta_t)=\int c(t,x,Y_t)\,dp(x,\zeta_t)\end{align}
we can state the following 

\begin{proposition}\label{WP1}
The separation property hinges upon how $x$ and $\zeta$ relate to each other via \eqref{a:h}\end{proposition}

\begin{proof} Recall first that the objective criterion in the exponentially transformed form is, for the separated problem, given at  \eqref{eq:Ibar}, namely
\begin{align}
\bar I(h;T,\theta,r_0)
= r_0^{-\theta} \bar{\mathbf{E}} \left[\exp \left\{ \theta \int_0^T \hat{g}(t, Z_t, h_t;\theta) dt\right\} \Psi^Z_T\right]
\end{align}
where, see \eqref{eq:gh},
\begin{align}\label{W1}
\hat g(t,z,h;\theta) 
&= \frac{\theta+1}{2} h'\Sigma^{(1)}\Sigma^{(1)'} \left(t, y\right) h 
- h' \hat{a}^{(1)}(t,z)
- \theta h'\Sigma^{(1)}\Xi'\left(t,y\right)
+ \hat{c}(t,z)
                                                \nonumber\\
&+ \frac{\theta -1}{2}\Xi\Xi'\left(t,y\right)
\end{align}
and, see \eqref{def:PbarZ},
\begin{align}\label{W2}
\Psi_t^Z 
=& \left(\bar{\chi}^Z_t\right)^{-1} = \exp \left\{ 
        \int_0^t \check{a}^{Y}(s, Z_s; h_s)' \left( \Sigma^Y\Sigma^{Y'}(s,Y_s)\right)^{-1}\Sigma^Y\left(s, Y_s\right) d\tilde{W}^h_s
\right. \nonumber\\
& \left.
       + \frac{1}{2}\int_0^t \check{a}^{Y}(s, Z_s; h_s)' \left( \Sigma^Y \Sigma^{Y'}\left(s, Y_s\right)\right)^{-1}\check{a}^Y(s,Z_s; h_s) ds \right\}
\end{align}
The function  $\check{a}^Y(t,Z_t;h_t)$ in \eqref{W2} is, see \eqref{eq:a_check}, given at
\begin{align}\label{W3} 
\check{a}^Y(t, Z_t; h_t) 
= \hat{a}^Y(t, Z_t) 
    - \theta \Sigma^Y\left(t, Y_t\right)\left(\Sigma^{(1)'}\left(t, Y_t\right) h_t - \Xi'\left(t, Y_t\right) \right)
\end{align}
It follows that the two ingredients of the objective criterion in the separated problem depend on $Z_t$ only via $\hat a^Y(t,Z_t)$ and $\hat c^Y(t,Z_t)$.
The conclusion then follows.
\end{proof}

\subsection{Separation Properties}\label{S.6.1}

We examine the implications of Proposition \ref{WP1} in four prominent cases. In the following, we discuss only properties relating to $a(t,x,Y_t)$ since they apply analogously to $c(t,x,Y_t)$.

\begin{enumerate}
\item In the \emph{linear} case with $a(t,x,y)=a(t,y)+A(t,y)\,x$ one has
\begin{align}
\hat a(t,\zeta_t,Y_t)=a(Y_t)+A(Y_t)\,m_t
\quad\mbox{with}\quad m_t = E\left[X_t\mid \mathcal{F}^Y_t\right],
\end{align}
where we made explicit $Z_t=(Y_t,\zeta_t)$. Thus, it follows from Proposition \ref{WP1} that the linear case $a(t,x,y)=a(t,y)+A(t,y)\,x$ is essentially the only case where a \emph{strict} separation property holds for our Gaussian model \eqref{eq:dS}, \eqref{eq:dL}. This finding is in line with the literature, in particular, \citet{nape02} who obtained an explicit analytic solution to the Modified Zakai Equation in the linear-Gaussian case with the use of the Kalman filter (KF). It also aligns with \citet{dall_BLcontinuous,dall_BBL_2020,davisRisksensitiveBenchmarkedAsset2021}.

\item  In the \emph{quadratic} case, $a(t,x,y)=a(t,y)\,x^2+b(t,x)\,x+c(t,y)$. Denoting the filter mean and co-variance by $m_t$ and $P_t$ respectively, we have 
\begin{align}
\hat a(t,\zeta_t,Y_t) = a(t,Y_t)\,[m_t^2+P_t]+b(t,Y_t)\,m_t+c(t,Y_t).
\end{align}
As is natural for the quadratic case, we have the additional additive term $a(t,Y_t)\,P_t$. From Proposition 6.1, we see that we cannot achieve strict separation by replacing $X_t$ by $m_t$ in the dynamics of the wealth process $V_t$ and excess return process $R_t$. However, we can express the separated problem with two synthetic values of the filter distribution, namely mean and covariance. Thus, the separation property holds in a \emph{wider sense}.

\item By extension, the Extended Kalman filter (EKF) performs a \emph{quadratic expansion} of a nonlinear $a(t,x,y)$ around their most recent estimate mean $m_t := m_t := \mathbf{E} \left[ X_t \mid \mathcal{F}^Y_t \right]$, i.e.
\begin{align} a(t,X_t, Y_t)=a(t,m_t,Y_t)+a_x(t,m_t,Y_t)'\,(X_t-m_t)+\frac{1}{2}\,(X_t-m_t)'\,a_{xx}(t,m_t,Y_t)\,(X_t-m_t)\end{align}
where $a_x$ and $a_{xx}$ denote partial derivatives, we have
\begin{align}
    \hat a(t,\zeta_t, Y_t)=a(t,m_t,Y_t)+\frac{1}{2}\, \tr(a_{xx}(t,m_t,Y_t)\, P_t).
\end{align}

\item Finally, the case $a(t,x,y)=\exp[\eta(t,y)\,x]$ subsumes a variety of nonlinear cases. Here, $\hat a((t,\zeta_t,Y_t)$ is the \emph{moment generating function} of the filter distribution $p(x,\zeta_t)$, provided it exists. In the case when one can apply the KF or, more generally, the EKF, the filter distribution is Gaussian and thus determined by the first two moments. One obtains
\begin{equation}\label{W4}
\hat a(t,\zeta_t,Y_t) = \exp\left[\eta(t,Y_t)\,m_t\right]\,\exp\left[\frac{1}{2}\eta^2(t,Y_t)\,P_t\right].
\end{equation}
In the simpler KF case, $P_t$ can be precomputed. So, the second term in \eqref{W4} is simply a multiplicative factor. In the EKF case, $P_t$ is adapted to the observation filtration, inducing a more significant difference than in the KF. Thus, in both cases, we can express the separated problem with two synthetic values of the filter distribution and achieve separation in a \emph{wider sense}.

\end{enumerate}

\section{Implementation Examples}\label{sec:examples}

We discuss some examples on how to use the results from Sections \ref{sec:Zakai} and \ref{sec:Separability}. We consider a general nonlinear Gaussian model as in Section \ref{sec:setting} where, according to Assumption \ref{as:FDfilter}, we assume that a finite-dimensional filter exists or can be derived via an approximation. To be concrete, we assume that such an approximation is obtained via the Extended Kalman Filter (EKF). The possibility then of obtaining an explicit analytic solution of the separated control problem, including an explicit expression for the candidate optimal control strategy, hinges upon the possibility of obtaining an, exact or approximate, analytic solution of a resulting HJB equation. As expected, this is certainly possible in the linear-Gaussian case for which an explicit solution is e.g. derived in Section 3.2 of \citet{davisRisksensitiveBenchmarkedAsset2021}. In this latter reference, the authors simply postulate a separation of estimation and control by stating that, since the investors cannot observe the value of certain factors $X(t)$, they use a modified risk-sensitive benchmarked criterion based on the Kalman filter estimate $\hat X(t)$. With our results, we know now that such a separation holds indeed in the linear-Gaussian case, and we are also able to draw some further conclusions, as mentioned below in this section.

\subsection{Approximate Solution via an Extended Kalman Filter}

Recall the partial observation model constructed in Section \ref{sec:setting}:
\begin{align}\label{eq:POmodel}
\begin{cases}
dX_t &= b\left(t,X_t, F_t\right)dt + \Lambda\left(t,X_t, F_t\right) dW_t,   \qquad X_0 \sim N\left( \mu_0, P_0\right)                                      \\
dY_t &= a^Y\left(t,X_t, F_t\right)dt + \Sigma^Y\left(t,F_t\right) dW_t,   \qquad Y_0 = y_0,
\end{cases}
\end{align}
where the distribution of the initial state value $X_0$ is Gaussian with mean $\mu$ and covariance $P_0$.

To set up the extended Kalman filter, we linearize the nonlinear functions of $X_t$ around their most recent estimates, i.e $m_t := \mathbf{E} \left[ X_t \mid \mathcal{F}^Y_t \right]$, namely
\begin{align}
    b\left(t,X_t, F_t\right) 
    &\approx b\left(t,m_t,F_t\right) + \frac{\partial b}{\partial X}\Big{|}_{\left(t,m_t,F_t\right)} 
    \left( X_t - m_t \right)
                                        \nonumber\\
    &= \left[ b(t,m_t,F_t) - \frac{\partial b}{\partial X}\Big{|}_{\left(t,m_t,F_t\right)} m_t\right] + \frac{\partial b}{\partial X}\Big{|}_{\left(t,m_t,F_t\right)} X_t
                                        \nonumber\\
    &=: \bar{b}\left(t,m_t,F_t\right) + B\left(t,m_t,F_t\right) X_t,
\end{align}
where we defined implicitly $\bar{b}(t,m_t,F_t) := b(t,m_t,F_t) - \frac{\partial b}{\partial X}\Big{|}_{\left(t,m_t,F_t\right)} m_t$ and $B(t,m_t,F_t) := \frac{\partial b}{\partial X}\Big{|}_{\left(t,m_t,F_t\right)}$. 
Similarly, 
\begin{align}
    a^Y\left(t,X_t, F_t\right) 
    &\approx a^Y\left(t,m_t,F_t\right) + \frac{\partial a^Y}{\partial X}\Big{|}_{\left(t,m_t,F_t\right)} 
    \left( X_t - m_t \right)
                                        \nonumber\\
    &= \left[ a^Y(t,m_t,F_t) - \frac{\partial a^Y}{\partial X}\Big{|}_{\left(t,m_t,F_t\right)} m_t\right] + \frac{\partial a^Y}{\partial X}\Big{|}_{\left(t,m_t,F_t\right)} X_t
                                        \nonumber\\
    &=: \bar{a}^Y\left(t,m_t,F_t\right) + A^Y\left(t,m_t,F_t\right) X_t,
\end{align}
where we defined implicitly $\bar{a}^Y \left(t,m_t,F_t \right) := a^Y(t,m_t,F_t) - \frac{\partial a^Y}{\partial X}\Big{|}_{\left(t,m_t,F_t\right)} m_t$ and $A^Y \left(t,m_t,F_t \right) := \frac{\partial a^Y}{\partial X}\Big{|}_{\left(t,m_t,F_t\right)}$. Taking also $\Lambda\left(t, X_t,F_t\right) \approx \Lambda\left(t, m_t,F_t\right)$, we have the following filter model:
\begin{align}\label{eq:POmodel:EKF}
\begin{cases}
dX_t &= \left[\bar{b}\left(t,m_t,F_t\right) + B\left(t,m_t,F_t\right) X_t \right]dt + \Lambda\left(t,m_t, F_t\right) dW_t,   \qquad X_0 \sim N\left( \mu_0, P_0\right)                                      \\
dY_t &= \left[\bar{a}^Y\left(t,m_t,F_t\right) + A^Y\left(t,m_t,F_t\right) X_t \right]dt + \Sigma^Y\left(t,F_t\right) dW_t,   \qquad Y_0 = y_0.
\end{cases}
\end{align}

The approximate filter distribution is Gaussian, with mean vector $m_t \in \mathbb{R}^{n}$ and covariance matrix $\Pi_t \in \mathbb{R}^{n \times n}$, so $\zeta_t = (m_t, \Pi_t)$. The filter mean
\begin{align}\label{eq:condmean:EKF}
    m_t := \mathbf{E} \left[ X_t \mid \mathcal{F}^Y_t \right]
\end{align}
satisfies the SDE
\begin{align}\label{eq:state_estimate:Zakai:sol:EKF}
    dm_t 
    =&
    \left[\bar{b}\left(t,m_t,F_t\right)+  B\left(t,m_t,F\right) m_t \right] dt 
                                            \nonumber\\
    & + \left[ \Pi_t A^{Y'}\left(t,m_t,F_t\right) + \Lambda\left(t,m_t,F_t\right)\Sigma^{Y'}\left(t,F_t\right) \right] \left(\Sigma^Y\Sigma^{Y'}\left(t,F_t\right)\right)^{-1} 
                                            \nonumber\\
    & \times \left[dY_t - \left(\bar{a}^Y\left(t,m_t,F_t\right) - A^Y\left(t,m_t,F_t\right) m_t \right) dt \right]
                                            \nonumber\\
    =& \left[\bar{b}\left(t,m_t,F_t\right)+  B\left(t,m_t,F\right) m_t \right] dt 
    + \hat{\Lambda}\left(t,m_t,\Pi_t,F_t\right) dU_t
                                                        \nonumber\\
    m_0 =& \mu_0,
\end{align}
where 
\begin{align*}
    \hat{\Lambda}\left(t,m_t,\Pi_t,F_t\right) := \left[\Pi_t A^{Y'}\left(t,m_t,F_t\right) + \Lambda\left(t,m_t,F_t\right)\Sigma^{Y'}\left(t,F_t\right) \right]\left(\Sigma^Y\Sigma^{Y'} (t, F_t)\right)^{-\frac{1}{2}},
\end{align*}
the process $U_t$, defined via
\begin{align}\label{eq:Kalmaninno:EKF}
  dU_t = \left(\Sigma^Y\Sigma^{Y'} \left(t,F_t\right)\right)^{-\frac{1}{2}} \left[dY_t -  \left(\bar{a}^Y\left(t,m_t,F_t\right) - A^Y\left(t,m_t,F_t\right) m_t \right) dt \right],  
\end{align}
is a $\left(\mathbb{P}, \mathcal{F}^Y_t\right)$-Brownian motion corresponding to the \emph{Kalman innovation}.

\begin{remark}
The implicit definition of the Kalman innovation $U_t$ at \eqref{eq:Kalmaninno:EKF} performs two actions simultaneously. It defines a new $\left(\mathbb{P}, \mathcal{F}^Y_t\right)$-Brownian motion $U_t$ in terms of the $\left(\mathbb{P}, \mathcal{F}_t\right)$-Brownian motion $W_t$, and it also ensures that the new Brownian motion has the same dimension as the observation process $Y_t$. Therefore, this implicit definition is tantamount to applying Lemma \ref{lem:Wiener:restricted} and performing the matrix transformation at \eqref{eq:rk:innovation:Kalmaninno} in Remark \ref{rk:innovation}.
\end{remark}

The filter covariance
\begin{align}\label{eq:condcovar:EKF}
    \Pi_t := \mathbf{E} \left[ \left(X_t - m_t \right) \left(X_t - m_t \right)' \mid \mathcal{F}^Y_t \right]
\end{align}
satisfies the Riccati equation
\begin{align}
\dot{\Pi}_t
    &=
        \Lambda\Lambda'\left(t,m_t,F_t\right)
        + B\left(t,m_t,F_t\right) \Pi_t
        + \Pi_t B'\left(t,m_t,F_t\right)
        - \hat{\Lambda}\left(t,m_t,\Pi_t,F_t\right)
          \hat{\Lambda}'\left(t,m_t,\Pi_t,F_t\right),
                                                    \nonumber\\
    \Pi_0 &= P_0.
\label{eq:Pi:Zakai:sol:EKF} 
\end{align}

We introduced the linear approximation solely for the purpose of deriving an approximate filter. The rest of the model remains nonlinear. Consequently, $\hat{a}^Y$ at \eqref{eq:ahat} becomes
\begin{align}\label{eq:aYhat:EKF} 
    \hat{a}^Y(t,Z)
    = \hat{a}^Y(t,m_t,\Pi_t, F)
    := \int  a^Y(t,X,F) d p(X;m_t,\Pi_t)
\end{align}
which, when decomposed, provides an expression for the securities' drift coefficient $\hat{a}$ and the benchmark's drift coefficient $\hat{c}$.

We mirror the development at \eqref{eq:Rtilde}, tayloring it to the situation at hand. Applying Lemma \ref{lem:Wiener:restricted}, and using the dynamics of $Y_t$ at \eqref{eq:POmodel}, the definition of $\hat{a}$ at \eqref{eq:aYhat:EKF}, and the block matrix decomposition in Remark \ref{rk:Wtilde:blocks}, we express $R_t$ in the filtration $\mathcal{F}^Y_t$, in terms of the filter mean $m_t$ and Wiener process $\tilde{W}_t$. Equation \eqref{eq:Rtilde} for the log return of the portfolio over its benchmark, $R_t = \ln \frac{V_t}{L_t}$, becomes:
\begin{align}\label{eq:R:EKF}
dR_t &= \left[ 
  \left(- \frac{1}{2} h_t'\Sigma^{(1)}\Sigma^{(1)'} \left(t, F_t\right) h_t + h_t'\hat{a}^{(1)}(t,m_t,\Pi_t,F_t) + \frac{1}{2}\Xi\Xi'\left(t, F_t\right) - \hat{c}(t,m_t,\Pi_t,F_t) \right) \right] dt
                                    \nonumber\\
    & + \left(h_t'\Sigma^{(1)}\left(t, F_t\right) - \Xi\left(t, F_t\right) \right) d\tilde{W}_t,     
\end{align}
with $R_0 = \ln \frac{v}{l} =: r_0$. 

For completeness, we also express $R_t$ in terms of the Kalman innovation $U_t$. Equation \eqref{eq:rk:innovation:Kalmaninno} implies that
\begin{align}\label{eq:innovation:Kalmaninno:EKF}
    \Sigma^{Y}\left(t, F_t\right) d\tilde{W}_t 
    = \left( \Sigma^{Y}\Sigma^{Y'}\left(t, F_t\right)\right)^{\frac{1}{2}} dU_t.   
\end{align}
We use the block matrix definition of $\Sigma^Y$ to decompose the left-hand side of this expression as $$\Sigma^{Y}\left(t, F_t\right) d\tilde{W}_t = 
\begin{pmatrix} 
    \Lambda^f\left(t,F_t\right)', &
    \Sigma\left(t,F_t\right)', &  
    \Xi\left(t,F_t\right)', & 
    \Sigma^{E}\left(t,F_t\right)' 
\end{pmatrix}'
d\tilde{W}_t.$$ 
On the right-hand side of \eqref{eq:innovation:Kalmaninno:EKF} we decompose the $m^Y \times m^Y$ matrix $\left(\Sigma^Y\Sigma^{Y'}\left(t, F_t\right)\right)^{1/2}$ as 
$$\left(\Sigma^Y\Sigma^{Y'}\left(t, F_t\right)\right)^{1/2} := 
\begin{pmatrix} 
    \hat{\Lambda}^f\left(t, F_t\right)', &
    \hat{\Sigma}\left(t, F_t\right)', & 
    \hat{\Xi}\left(t, F_t\right)' &  
    \hat{\Psi}_Z\left(t, F_t\right)'
\end{pmatrix}',
$$ 
where $\hat{\Lambda}^f(\cdot)$, $\hat{\Sigma}(\cdot)$, $\hat{\Xi}(\cdot)$, and $\hat{\Psi}_Z(\cdot)$ are respectively a $\ell \times m^Y$ matrix, a $m \times m^Y$ matrix, a $m^Y$-element column vector, and a $k \times m^Y$ matrix such that $\hat{\Lambda}^f\hat{\Lambda}^{f'}(\cdot) = \Lambda^f\Lambda^{f'}(\cdot)$, $\hat{\Sigma}\hat{\Sigma}'(\cdot) = \Sigma\Sigma'(\cdot)$, $\hat{\Xi}\hat{\Xi}'(\cdot) = \Xi\Xi'(\cdot)$, and $\hat{\Psi}_Z\hat{\Psi}_Z'(\cdot) = \Psi_Z\Psi_Z'(\cdot)$.
Using these decompositions, we restate \eqref{eq:innovation:Kalmaninno:EKF} as
\begin{align}
&
\begin{pmatrix} 
    \Lambda^f\left(t,F_t\right)', &
    \Sigma\left(t,F_t\right)', &  
    \Xi\left(t,F_t\right)', & 
    \Sigma^{E}\left(t,F_t\right)'
\end{pmatrix}'d\tilde{W}_t
                                    \nonumber\\
=&
\begin{pmatrix}     
    \hat{\Lambda}^f\left(t, F_t\right)' ,&
    \hat{\Sigma}\left(t, F_t\right)', & 
    \hat{\Xi}\left(t, F_t\right)', &  
    \hat{\Psi}_Z\left(t, F_t\right)'
\end{pmatrix}'dU_t, 
\end{align}
which implies in particular that $\Sigma\left(t,F_t\right) d\tilde{W}_t = \hat{\Sigma}\left(t, F_t\right) dU_t$ and $\Xi\left(t,F_t\right) d\tilde{W}_t = \hat{\Xi}\left(t, F_t\right) dU_t$. We use these two equations and a further decomposition of $\Sigma$ and $\hat{\Sigma}$ distinguishing between tradable and nontradable assets to rewrite \eqref{eq:R:EKF} as 
\begin{align}\label{eq:R:EKF:U}
dR_t &= \left[ 
  \left(- \frac{1}{2} h_t'\Sigma^{(1)}\Sigma^{(1)'} \left(t, F_t\right) h_t + h_t'\hat{a}^{(1)}(t,m_t,\Pi_t,Y_t) + \frac{1}{2}\Xi\Xi'\left(t, F_t\right) - \hat{c}(t,m_t,\Pi_t,F_t) \right) \right] dt
                                    \nonumber\\
    & + \left(h_t'\hat{\Sigma}^{(1)}\left(t, F_t\right) - \hat{\Xi}\left(t, F_t\right) \right) dU_t.   
\end{align}

Then, the risk-sensitive benchmarked criterion at \eqref{eq:Jhat} becomes
\begin{align}\label{eq:Jhat:EKF}
\hat J(h;T,\theta,r_0) 
&:= -\frac{1}{\theta} \ln \mathbf{E} \left[ r_0^{-\theta} e^{-\theta R_T}\right]
= -\frac{1}{\theta} \ln \mathbf{E} \left[    
r_0^{-\theta} \exp \left\{ \theta \int_{0}^{T} \hat g(t,m_t, \Pi_t, F_t, h_t;\theta) dt \right\} \hat{\chi}_T^h
\right],
\end{align}
where 
\begin{align}\label{eq:g:EKF}
\hat{g}(t,m,\Pi,f,h;\theta)
:=&
\frac{1}{2} \left(\theta+1 \right)h'\Sigma^{(1)}\Sigma^{(1)'}(t, f) h 
- h' \hat{a}^{(1)} \left(t, m, \Pi, f\right)
- \theta h'\Sigma^{(1)}\Xi'(t, f)
                                                \nonumber\\
& 
+ \hat{c} \left(t, m,\Pi,f\right)
+ \frac{1}{2} \left(\theta -1 \right)\Xi\Xi'(t,f),
\end{align}
and $\hat{\chi}_T^h$ is defined via \eqref{eq:expmarthat}. 

\begin{remark}
As expected, 
$\hat{g}(t,m_t,\Pi_t,f,h;\theta)
=
\int g(t,X,F,h;\theta) dp(X;m_t, \Pi_t)$.
\end{remark}

Theorem \ref{tequiv} guarantees that, for a given admissible control policy $h_t \in \mathcal{A}_Z(T)$, the separated criterion $\hat J(h;T,\theta,r_0)$ at \eqref{eq:Jhat:EKF} yields the same value as the original criterion. Thus, we have reformulated the control problem with a possible approximation in the derivation of the finite-dimensional filter, using only the properties of the finite-dimensional filter, without needing to derive and solve the Zakai equation.

\subsection{Linear-Gaussian Model}

The linear-Gaussian model is a corollary to the EKF approximation described above. Here, the drift functions $b(t,X_t,F_t)$, $b^f(t,X_t,F_t)$, $a(t,X_t,F_t)$, $c(t,X_t,F_t)$, $a^E(t,X_t,F_t)$ are all affine in the factor processes $X_t$ and $F_t$, and the diffusion coefficients $\Lambda$, $\Sigma$, $\Xi$, and $\Sigma^E$ are deterministic functions of time.  Specifically, $b(t,X_t,F_t) = b_t + B_t \begin{pmatrix} X_t \\ F_t \end{pmatrix}$, where the time-dependent coefficients $b : [0,T] \to \mathbb{R}^n, B: [0,T] \to \mathbb{R}^{m\times (n+\ell)}$, are $C^1$, with similar assumptions for $b^f(t,X_t,F_t)$, $a(t,X_t,F_t)$, $c(t,X_t,F_t)$, and $a^E(t,X_t,F_t)$. Additionally, $X_0 \sim N(\mu_0, P_0)$ for some known vector $\mu_0 \in \mathbb{R}^n$ and matrix $P_0 \in \mathbb{R}^{n \times n}$. The dynamics of the observation process $Y_t = (F_t', \ln S_t', \ln L_t', E_t')'$ is:
\begin{align}\label{eq:obs:LG}
dY_t &= \left( a^Y_t + A^Y_t \begin{pmatrix} X_t \\ F_t \end{pmatrix} \right)dt + \Sigma^Y_t dW_t, \; Y_0 = y_0,
\end{align}
Hence, the filter model \eqref{eq:POmodel:EKF} holds exactly.

Then,
\begin{align}\label{eq:ahat:LG} 
    \hat{a}^Y(t,Z)
    = \int  a^Y(t,X,Y) d p(X;m_t,\Pi_t)
    = a^Y_t + A^Y_t \begin{pmatrix}
                m_t     \\
                F_t
        \end{pmatrix}.
\end{align}

The risk-sensitive benchmarked criterion at \eqref{eq:Jhat} becomes
\begin{align}\label{eq:Jhat:LG}
\hat J(h;T,\theta,r_0) 
&:= -\frac{1}{\theta} \ln \mathbf{E} \left[ r_0^{-\theta} e^{-\theta R_T}\right]
= -\frac{1}{\theta} \ln \mathbf{E} \left[    
r_0^{-\theta} \exp \left\{ \theta \int_{0}^{T} \hat g(t,m_t, F_t, h_t;\theta) dt \right\} \hat{\chi}_T^h
\right],
\end{align}
where 
\begin{align}\label{eq:g:LG}
\hat{g}(t,m,f,h;\theta)
&:= \frac{\theta+1}{2} h'\Sigma_t^{(1)}\Sigma_t^{(1)'} h 
- h' \left[ a^{(1)}_t + A^{(1)}_t \begin{pmatrix}
                m     \\
                F
        \end{pmatrix} \right]
- \theta h'\Sigma_t^{(1)}\Xi_t'
+ \left[ c_t + C_t \begin{pmatrix}
                m     \\
                F
        \end{pmatrix} \right]
                                        \nonumber\\
&+ \frac{\theta -1}{2}\Xi_t\Xi_t'
\end{align}
and $\hat{\chi}_T^h$ is defined via \eqref{eq:expmarthat}, as before.

Theorem \ref{tequiv} guarantees that the separated criterion $\hat J(h;T,\theta,r_0)$ at \eqref{eq:Jhat:LG} is identical to the original criterion for the linear-Gaussian control problem. Therefore, we have already achieved the same result as Theorem 2.1 in \citet{na00}, without deriving and solving a Zakai equation. All that is left is to solve the control problem as in \citet{davisRisksensitiveBenchmarkedAsset2021}.

\section{Conclusion}
This paper investigated stochastic control problems, in particular of the risk-sensitive type, under incomplete observation. We were mainly inspired by Nagai and Peng \citep{na00,nape02} whose approach, based on a so-called Modified Zakai Equation (MZE), solves the incomplete observation stochastic control problem in the linear-Gaussian case. We considered a general nonlinear model inspired by the risk-sensitive benchmarked asset management model (RSBAM) in  \citet{davisRisksensitiveBenchmarkedAsset2021} for which we assumed the existence of a finite-dimensional filter described by a vector-parameter process $\zeta_t$. The original problem with the unobserved state process $X_t$ was then transformed into the so-called separated problem with the state variable process given by $\zeta_t$.\smallskip

Our {\bf contribution} can be seen as follows:
\begin{enumerate}
    \item[a)] Theorem \ref{tequiv} proves the equivalence of the given problem in its original and separated versions, a fact that is generally overlooked in the literature;
    
    \item[b)] Proposition \ref{WP1} provides a criterion allowing us to identify situations where a separation property holds, perhaps in a wider sense, namely when the separated problem can be expressed in terms of just a finite number of synthetic values of the filter distribution such as the mean and some higher order moments. The state variable process in the separated problem is then finite-dimensional. Therefore, the solution to the separated problem can be obtained by standard methods of completely observed stochastic control problems;
    
    \item[c)] Our derivation of the MZE, given in Theorem \ref{theo:Zakai:Z}, aligns with Nagai and Peng, but with one crucial difference. While Nagai and Peng let the filter intervene only at the level of the solution of the MZE, we take the filter into account from the very beginning thanks to Lemma \ref{lem:Wiener:restricted}. Consequently, the MZE in Nagai and Peng is a stochastic PDE, which is difficult to solve in general. In fact, they obtain an explicit solution only in the linear-Gaussian case thereby implicitly establishing a separation property for this special case. However, in our situation, the MZE reduces to a deterministic PDE. Therefore, when solving the partially observed stochastic control problem by solving this deterministic PDE, one ends up with the same degree of difficulty as when solving an HJB equation for a complete observation stochastic control problem. Furthermore, our approach works for general incomplete observation stochastic control problems, provided there exists a finite-dimensional filter. 
\end{enumerate}

\appendix
\renewcommand{\thesection}{\Alph{section}}

\section{Relation to Existing Risk-Sensitive Investment Management Problems}\label{app:correspondence}

Our approach applies directly to various existing models such as those by \citet{na00}, \citet{nape02}, and \citet{dall_BLcontinuous, dallBBL_JPM_2016, dall_BBL_2020, davisRisksensitiveBenchmarkedAsset2021}. Table \ref{table:parms:mapping} details the parametrization of these models in our general setup. Therefore, the approach describes in our paper establishes the separability of these models, and confirms that the asset allocation derived in \citet{dall_BLcontinuous, dallBBL_JPM_2016, dall_BBL_2020, davisRisksensitiveBenchmarkedAsset2021} is optimal.

\begin{table}[h!]
\begin{center}
\resizebox{1.00\textwidth}{!}{%
\begin{tabular}{| l | c | c | c | c |}
\hline
    & In this paper 
    & \citet{na00, nape02}    
    &   \citet{dall_BLcontinuous,dallBBL_JPM_2016,dall_BBL_2020} 
    &   \citet{davisRisksensitiveBenchmarkedAsset2021}  \\
    &&&&
    \\
\hline\hline
    Type of criterion 
    & Benchmark outperformance
    & Wealth-based  
    & Wealth-based  
    & Benchmark outperformance \\
\hline
    Unobservable factor  $X_t$
    & $b(t, X_t, F_t)$  
    & $b + B X_t$
    & $b + B X_t$
    & $b_t + B_t X_t$   \\
    & $\Lambda(t, X_t, F_t)$ 
    & $\Lambda$
    & $\Lambda$
    & $\Lambda_t$ \\
\hline
    Observable factor  $F_t$
    & $b^f(t, X_t, F_t)$, 
    & -
    & -
    & - \\
    & $\Lambda^f(t, F_t)$
    & -
    & -
    & - \\
\hline
    Financial assets  $S_t$
    & $a(t, X_t, F_t)$  
    & $a + A X_t$
    & $a + A X_t$
    & $a_t + A_t X_t$   \\
    & $\Sigma(t, F_t)$ 
    & $\Sigma$
    & $\Sigma$
    & $\Sigma_t$ \\
\hline
    Benchmark $L_t$
    & $c(t, X_t, F_t)$  
    & -
    & -
    & $c_t + C_t X_t$   \\
    & $\Xi(t, F_t)$ 
    & -
    & -
    & $\Xi_t$ \\
\hline
    Expert forecasts $L_t$
    & $a^E(t, X_t, F_t)$  
    & -
    & $a^E_t + A^E_t X_t$
    & $a^E_t + A^E_t X_t$   \\
    & $\Sigma^E(t, F_t)$ 
    & -
    & $\Sigma^E_t$
    & $\Sigma^E_t$ \\
\hline
    Sources of observations 
    & Financial assets
    & Financial assets   
    & Financial assets 
    & Financial assets    \\
    & Benchmark
    &    
    & Expert forecasts 
    & Benchmark    \\
    & Expert forecasts  
    &           
    &   
    & Expert forecasts    \\
    & Observable factors  
    &           
    &   
    & \\
    &&&&
    \\   
    Observation process $Y_t$
    & $a^Y(t, X_t, F_t)
    = \begin{pmatrix} 
        b^f(t,X_t,F_t) \\
        a(t,X_t,F_t) - \frac{1}{2}d_{\Sigma}(t,F_t) \\
        c_t - \frac{1}{2} \Xi\Xi_t'  \\ 
        a^E_t 
    \end{pmatrix}$
    & $a - \frac{1}{2}d_{\Sigma} + A X_t$ 
    & $\begin{pmatrix} a - \frac{1}{2}d_{\Sigma}\\ a^E_t \end{pmatrix} + \begin{pmatrix} A \\ A^E_t \end{pmatrix} X_t$
    & $\begin{pmatrix} 
        a_t - \frac{1}{2}d_{\Sigma}(t) \\ 
        c_t - \frac{1}{2} \Xi\Xi_t'  \\ 
        a^E_t 
    \end{pmatrix}
    + \begin{pmatrix} A_t \\ C_t \\ A^E_t \end{pmatrix} X_t$
    \\
    &&&&
    \\
    & $\Sigma^Y(t, F_t)$ 
        = $\begin{pmatrix} \Lambda^f(t,F_t)
        \\ \Sigma (t,F_t) 
        \\ \Xi(t,F_t) 
        \\ \Sigma^E (t,F_t)) \end{pmatrix}$
    & $\Sigma$
    & $\begin{pmatrix} \Sigma \\ \Sigma^E_t \end{pmatrix}$
    & $\begin{pmatrix} \Sigma_t \\ \Xi_t \\ \Sigma^E_t) \end{pmatrix}$
    \\
\hline
\end{tabular}%
}
\end{center}
\caption{Correspondence between our model and existing risk-sensitive investment models with partial observation.}
\label{table:parms:mapping}
\end{table}

\newpage

\section{Proof of Proposition \ref{prop:P_h:unique} }\label{app:proofs:controlpb:equival}

\textit{Proof of} $i)$. We start by proving that condition \eqref{def:control:class:AZ:Kaza1} in Definition \ref{def:control:class:AZ} implies condition \eqref{def:control:class:AX:Kaza1} in  Definition \ref{def:control:class:AX}. 
    
    By Lemma \ref{lem:Wiener:restricted}, 
    $
        d\tilde{W}_t = \Sigma^{Y'}\left(t, Y_t\right) \left(\Sigma^Y \Sigma^{Y'}\left(t, Y_t\right)\right)^{-1} \left( \Sigma^Y\left(t, Y_t\right) dW_t + a^Y(t,X_t,Y_t) - \hat{a}^Y(t,Z_t) dt \right),
    $
    thus, condition \eqref{def:control:class:AZ:Kaza1} implies that
    {\small
    \begin{align}\label{eq:prop:P_h:unique:interm1}
        & \mathbf{E} \left[ \exp \left\{ -\frac{\theta}{2} \int_{0}^{t}\left(h_s' \Sigma^{(1)}\left(s, Y_s\right) - \Xi\left(s, Y_s\right) \right)d \tilde{W}_s\right\} \right] 
                                                                    \nonumber\\
        =& \mathbf{E} \left[ \exp \left\{ -\frac{\theta}{2} \int_{0}^{t}\left(h_s' \Sigma^{(1)}\left(s, Y_s\right) - \Xi\left(s, Y_s\right) \right) 
        \right.\right.                                               \nonumber\\
        & \left.\left.\times   
        \left( \Sigma^{Y'}\left(s, Y_s\right) \left(\Sigma^Y \Sigma^{Y'}\left(s, Y_s\right)\right)^{-1} \left( \Sigma^Y\left(s, Y_s\right) dW_t + a^Y(s,X_s,Y_s) - \hat{a}^Y(s,Z_s) ds \right)\right)\right\} \right] < \infty.        
    \end{align}
    }
    
    Next, we simplify the expression inside the expectation at \eqref{eq:prop:P_h:unique:interm1}. Let $\tilde{h}_t $ be the $m^Y$-element vector process according to Notation \ref{N.2}. Define also the $m^Y$-element column indicator vector
    $
    \mathbf{1}_{\Xi} =
    \begin{pmatrix}
        0_{\ell + m}    &
        1               &
        0_{k},
    \end{pmatrix}'
    $,
    that is, $\mathbf{1}_{\Xi}$ is a vector with entry $(\ell + m + 1)$ set to 1 and all other entries set to 0. Note that $\mathbf{1}_{\Xi}'\Sigma^{Y}_t  = \Xi_t$. Then, we express \eqref{eq:prop:P_h:unique:interm1} as:
    {\small
    \begin{align}
        & \mathbf{E} \left[ \exp \left\{ -\frac{\theta}{2} \int_{0}^{t}\left(\tilde{h}_s'\Sigma^{Y}\left(s, Y_s\right)  - \mathbf{1}_{\Xi}'\Sigma^{Y}\left(s, Y_s\right) \right)
        \right.\right.                                               \nonumber\\
        & \left.\left.\times   
        \left( \Sigma^{Y'}\left(s, Y_s\right) \left(\Sigma^Y \Sigma^{Y'}\left(s, Y_s\right)\right)^{-1} \left( \Sigma^Y\left(s, Y_s\right) dW_t + a^Y(s,X_s,Y_s) - \hat{a}^Y(s,Z_s) ds \right)\right)\right\} \right]
                                                                    \nonumber\\
        =& \mathbf{E} \left[ \exp \left\{ -\frac{\theta}{2} \int_{0}^{t}\left(\tilde{h}_s'  - \mathbf{1}_{\Xi}' \right)
        \cancel{\left(\Sigma^{Y}\Sigma^{Y'}\left(s, Y_s\right)\right)}\cancel{\left(\Sigma^Y \Sigma^{Y'}\left(s, Y_s\right)\right)^{-1}}
                \right.\right.                                               \nonumber\\
        & \left.\left.\times  
        \left( \Sigma^Y\left(s, Y_s\right) dW_t + a^Y(s,X_s,Y_s) - \hat{a}^Y(s,Z_s) ds \right)\right\} \right]
                                                                    \nonumber\\
        =& \mathbf{E} \left[ \exp \left\{ -\frac{\theta}{2} \int_{0}^{t}\left(h_s' \Sigma^{(1)}\left(s, Y_s\right) - \Xi\left(s, Y_s\right) \right) dW_s \right\}
        \right.                                               \nonumber\\
        & \left.\times  
        \exp \left\{ -\frac{\theta}{2} \int_{0}^{t}\left(\tilde{h}_s'  - \mathbf{1}_{\Xi}' \right) 
        \left(a^Y(s,X_s,Y_s) - \hat{a}^Y(s,Z_s) ds \right)\right\} \right] < \infty.
                                                                    \nonumber
        \end{align}
    }
    
    The first exponential is finite as long as the second exponential is strictly positive.  Recall that $a^Y$ is $C^{1,1,1}_b$ on $[0,T] \times \mathbb{R}^n \times \mathbb{R}^m_Y$, so the conditional expectation $\hat{a}$ is also bounded, and
    {\small
    \begin{align}
        \exp \left\{ -\frac{\theta}{2} \int_{0}^{t}\left(h_s'  - \mathbf{1}_{\Xi}' \right) 
        \left(a^Y(s,X_s,Y_s) - \hat{a}^Y(s,Z_s) ds \right)\right\} > 0
        \qquad \forall (s,x,y) \text{ a.s.}
    \end{align}
    }
    Consequently, 
    $
        \mathbf{E} \left[ \exp \left\{ -\frac{\theta}{2} \int_{0}^{t}\left(h_s' \Sigma^{(1)}\left(s, Y_s\right) - \Xi\left(s, Y_s\right) \right) dW_s \right\}
        \right] < \infty,      
    $    
    and the Kazamaki condition \eqref{def:control:class:AX:Kaza1} in Definition \ref{def:control:class:AX} holds. Similar reasoning, starting from condition \eqref{def:control:class:AX:Kaza1} in  Definition \ref{def:control:class:AX} and applying Lemma \ref{lem:Wiener:restricted} shows that the converse is true. Therefore, we have proved the equivalence of conditions \eqref{def:control:class:AZ:Kaza1} and \eqref{def:control:class:AX:Kaza1}.  The proof of the equivalence of conditions \eqref{def:control:class:AZ:Kaza2}  and  \eqref{def:control:class:AX:Kaza2} proceeds analogously.

\noindent\textit{Proof of} $ii)$. Under the measure $\mathbb{P}^h$ defined at \eqref{def:Ph}, the $\mathcal{F}_t$-standard Wiener process is (see \eqref{eq:Ph:BM:Wh:X})   
    {\small
    \begin{align*}
        W^{h}_t 
        := W_t 
        + \theta \int_{0}^t \left(\Sigma^{(1)'}\left(s, Y_s\right) h_s - \Xi'\left(s, Y_s\right) \right) ds,
    \end{align*}
    }
    for $h_t \in \mathcal{A}^X(T)$. By Lemma \ref{lem:Wiener:restricted}, there exists a $\left(\mathbb{P}^h,\mathcal{F}^Y_t\right)$ $d$-dimensional standard Wiener process $\tilde{W}^X_t$ such that
    $
        \Sigma^Y\left(t, Y_t\right) d\tilde{W}_t^X 
        = \left(a^Y(t,X_t,Y_t) - \hat{a}^Y(t,Z_t) \right)dt + \Sigma^Y\left(t, Y_t\right) dW^h_t.
    $
    Hence,
    {\small
    \begin{align}\label{eq:P_h:unique:eq1}
        \Sigma^Y\left(t, Y_t\right) d\tilde{W}_t^X 
        =& \left(a^Y(t,X_t,Y_t) - \hat{a}^Y(t,Z_t) \right)dt + \Sigma^Y\left(t, Y_t\right) dW_t 
        + \theta \Sigma^Y\left(t, Y_t\right)\left(\Sigma^{(1)}\left(t, Y_t\right) h_t - \Xi\left(t, Y_t\right) \right)'.
    \end{align}
    }
    However, applying Lemma \ref{lem:Wiener:restricted} under the measure $\mathbb{P}$ also tells us that 
    {\small
    \begin{align}   
        \Sigma^Y\left(t, Y_t\right) dW_t  
        = \left(\hat{a}^Y(t,Z_t) - a^Y(t,X_t,Y_t) \right)dt + \Sigma^Y\left(t, Y_t\right) d\tilde{W}_t   ,                              \nonumber\\
    \end{align}
    }    
    so \eqref{eq:P_h:unique:eq1} becomes 
    {\small
    \begin{align}
        \Sigma^Y\left(t, Y_t\right) d\tilde{W}^X_t 
        &= \Sigma^Y\left(t, Y_t\right) d\tilde{W}_t + \theta \Sigma^Y\left(t, Y_t\right)\left(\Sigma^{(1)'}\left(t, Y_t\right) h_t - \Xi'\left(t, Y_t\right) \right),
    \end{align}
    }
    which is the definition of the standard $(\hat{\mathbb{P}}^h, \mathcal{F}^Y_t)$-Wiener process $\tilde{W}_t^h$ at \eqref{eq:Ph:BM:What}. Hence, $\tilde{W}_t^X = \tilde{W}_t^h$, so $\tilde{W}_t^h$ is also the restriction of the standard $(\mathbb{P}^h, \mathcal{F}_t)$-Wiener process $W_t^h$ to the filtration $\mathcal{F}_t^Y$. As Lemma \ref{lem:Wiener:restricted}, affects only the filtration but not the measure, we conclude that the measures $\mathbb{P}^h$ defined via \eqref{def:Ph} and $\hat{\mathbb{P}}^h$ defined via \eqref{def:Phhat} are identical.

\section{The MZE approach by Nagai and Peng}\label{app:MZE:NagaiPeng}

To recall the approach by Nagai and Peng we have to start from the problem formulation as in Subsection \ref{S.4.1} with the unobserved factor process $X_t$. \citet{na00} finds it convenient to have $Y_t$ a martingale by passing to a new measure $\bar{\mathbb{P}}$. Assuming $h_t \in \mathcal{A}_X(T)$, define $\bar{\mathbb{P}}$ via the Radon-Nikodym derivative $\bar\chi_t^X  = \frac{d\bar{\mathbb{P}}}{d\mathbb{P}^h} \Big{|}_{\mathcal{F}_t}$, where
\begin{multline}\label{def:PbarX}
    \bar\chi_t^X 
    := \exp \left\{ 
        - \int_0^t a^{\dagger}(s, X_s,Y_s; h_s)' \left( \Sigma^Y\Sigma^{Y'}(s,Y_s)\right)^{-1}\Sigma^Y\left(s, Y_s\right) dW^h_s
                        \right. \\
       \left.
       - \frac{1}{2}\int_0^t a^{\dagger}(s, X_s,Y_s; h_s)' \left( \Sigma^Y \Sigma^{Y'}\left(s, Y_s\right)\right)^{-1} a^\dagger(s, X_s,Y_s; h_s) ds \right\}
\end{multline}
with $a^\dagger$ defined at \eqref{eq:a_dag}, and for $t \in [0,T]$ so that the standard $(\bar{\mathbb{P}},\mathcal{F}_t)$-Wiener process is
\begin{align}\label{eq:Wiener:PbarX}
    \bar{W}_t 
    := W^{h}_t 
    + \int_{0}^t \Sigma^{Y'}\left(s, Y_s\right)\left( \Sigma^Y\Sigma^{Y'}\left(s, Y_s\right)\right)^{-1} a^\dagger(s, X_s,Y_s; h_s) ds
\end{align}
implying a corresponding $\bar{\mathbb{P}}$-dynamics for $(X_t, Y_t)$.  The risk-sensitive and exponentially transformed criteria is then
\begin{align}\label{eq:JbarX}
    \bar{J}(h;T,\theta,r_0) =  -\frac{1}{\theta} \ln \bar{\mathbf{E}} \left[ r_0^{-\theta} \exp \left\{ \theta \int_0^T g(t, X_t, Y_t, h_t;\theta) dt\right\} \Psi_T^X\right],
\end{align}
where $\Psi_t^X:=\left(\bar\chi_t^X\right)^{-1}$  is a martingale on $(\bar{\mathbb{P}}, \mathcal{F}_t)$, corresponding to the Radon-Nikodym derivative $\Psi_t^X = \frac{d\mathbb{P}^h}{d\bar{\mathbb{P}}} \Big{|}_{\mathcal{F}_t}$.
We also define the usual exponentially transformed criterion $\bar{I}(h;T,\theta,r_0) := e^{-\theta \bar{J}(h;T,\theta,r_0)}$.

By the Tower Property,
\begin{align}\label{eq:Ibar:TPX}
    \bar{I}(h;T,\theta,r_0) 
    &= r_0^{-\theta} \bar{\mathbf{E}} \left[ \bar{\mathbf{E}} \left[\exp \left\{ \theta \int_0^T g(t, X_t, Y_t, h_t;\theta) dt\right\} \Psi^X_T \mid \mathcal{F}^Y_T \right]\right],    
\end{align}

where $g(\cdot)$ is as in \eqref{eq:g}. 

\begin{definition}\label{def:qX}
For $h_t \in \mathcal{A}_X(T)$, and $\varphi(t,x;y): [0,T] \times \mathbb{R}^{n} \times \mathbb{R}^{m^Y} \to \mathbb{R}$ a $C_b^{1,2}$ test function indexed by $y$, set
\begin{align}\label{eq:qX}
    q_X^h(t)(\varphi_t) := \bar{\mathbf{E}} \left[e^{ \theta \int_0^t g(s,X_s, Y_s,h_s;\theta)ds} \Psi^X_t \varphi(t, X_t; Y_t) \mid \mathcal{F}^Y_t \right].
\end{align}
\end{definition}
From this definition, we immediately see that
$
    \bar{I}(h;T,\theta,r_0) 
    = r_0^{-\theta} \bar{\mathbf{E}} \left[ q_X^h(T)(1) \right].  
$ \citet{na00} now proves the following

\begin{theorem}[Classical formulation of the Modified Zakai Equation (Nagai 2001)]\label{theo:Zakai:X}
The operator $q_X^h(t)(\varphi_t)$ in Definition \ref{def:qX} satisfies the stochastic partial differential equation (SPDE)
\begin{align}\label{eq:Zakai:SPDE:X}
    & q_X^h(t)\left(\varphi(t)\right)
                                                                    \nonumber\\
    =& q_X^h(0)\left(\varphi(0)\right)
                                                                    \nonumber\\
    &   + \int_0^t q_X^h(s)\left(\frac{\partial \varphi}{\partial t}(s,X_s;Y_s)
            + L_X^h\varphi(s,X_s;Y_s)
            + \theta g(s,X_s,Y_s,h_s;\theta)\varphi(s,X_s;Y_s)\right)ds
                       \nonumber\\            
    &   + \int_0^t q_X^h(s)\left( L_Y^h \varphi(s,X_s;Y_s)\right) dY_s,
\end{align}
where 
\begin{align}\label{eq:operator:Lh:X}
    L_X^h\varphi(t,x;y) 
    :=& \frac{1}{2} \tr \left\{ \Lambda\Lambda'(t,x;y) D^2\varphi \right\}
    + \left[b(t,x;y) 
        - \theta \Lambda(t,x;y) \left(\Sigma^{(1)'}\left(t, y\right) h - \Xi'\left(t, y\right) \right)
        \right]' D\varphi
                                                                    \nonumber\\
    =& \frac{1}{2} \sum_{i,j=1}^n \left\{\Lambda\Lambda'(t,x;y)\right\}_{ij} \frac{\partial^2 \varphi}{\partial x_i \partial x_j}
                                                                    \nonumber\\
    &+ \sum_{i=1}^n\left[b(t,x;y) 
        - \theta \Lambda(t,x;y) \left(\Sigma^{(1)'}\left(t, y\right) h - \Xi'\left(t, y\right) \right)
        \right]_i \frac{\partial \varphi}{\partial x_i}
\end{align}
is the $\mathbb{P}^h$-generator related to the process $X_t$, $D\varphi = \begin{pmatrix} \frac{\partial \varphi}{\partial x_1} & \ldots & \frac{\partial \varphi}{\partial x_i} & \ldots & \frac{\partial \varphi}{\partial x_n} \end{pmatrix}$ and $D^2\varphi = \left[ \frac{\partial^2 \varphi}{\partial x_i x_j} \right], i,j = 1, \ldots, n$. The operator $L_Y^h\varphi(t,x;y)$ is defined in vector form as
\begin{align}\label{eq:operator:Lh:Y}
    L_Y^h\varphi(t,x;y) 
    :=& \left[ a^Y(t, x; y) - \theta \Sigma^Y\left(t; y\right) \left(\Sigma^{(1)'}\left(t; y\right) h - \Xi'\left(t; y\right) \right)\right]' 
        \left( \Sigma^Y\Sigma^{Y'}\left(t, y\right)\right)^{-1} \varphi(t,x;y)
                       \nonumber\\            
    &  + D'\varphi(t,x;y) \Lambda(t,x;y) \Sigma^{Y'}\left(t; y\right) \left(\Sigma^Y \Sigma^{Y'}\left(t, y\right)\right)^{-1},
\end{align}
or element-by-element as:
\begin{align}
    \left\{L_Y^h \varphi(t,x;y)\right\}_i 
    :=& \left\{\left[ a^Y(t, x; y) - \theta \Sigma^Y\left(t; y\right) \left(\Sigma^{(1)'}\left(t; y\right) h - \Xi'\left(t; y\right) \right)\right]' 
        \left( \Sigma^Y\Sigma^{Y'}\left(t, y\right)\right)^{-1}\right\}_i \varphi(t,x;y)
                       \nonumber\\         
    &+ \sum_{j=1}^n \frac{\partial \varphi(t,x;y)}{\partial x_j} \left\{\Lambda(t,x;y)\Sigma^{Y'}\left(t; y\right) \left(\Sigma^Y \Sigma^{Y'}\left(t, y\right)\right)^{-1}\right\}_{ji}                    
    \qquad i = 1, \ldots, m^Y.
\end{align}
\end{theorem}

\begin{remark}\label{rk:SPDE:operators}
Note that although $q_X^h$ is defined as an expectation under the measure $\bar{\mathbb{P}}$, the drift in the SPDE of Theorem \ref{theo:Zakai:X} depends on the $\mathbb{P}^h$-operator $ L^h_X$. How did we get there? The derivation starts with the $\bar{\mathbb{P}}$-operator of $X_t$. This operator then acquires an extra term through the cross variation $d \langle \Psi^X, \varphi \rangle_t$, yielding the $\mathbb{P}^h$-operator $L^h_X$.
\end{remark}

\begin{remark}\label{rk:Zakai_as_a_standard_SPDE}
We can express the conditional probability $q_X^h(t)(\varphi_t)$ in terms of its conditional density $q_X^h(t,x;y)$ as
\begin{align}
    q_X^h(t)\left(\varphi(t)\right) = \int q_X^h(t,x;y) e^{ \theta \int_0^t g(s,x, y,h;\theta)ds} \varphi(t,x;y) dx,
\end{align}
where the $q_X^h(t,x;y)$ density is with respect to the Lebesgue measure. By standard arguments, we can show that the conditional density $q_X^h(t)$ satisfies the stochastic partial differential equation (SPDE):
\begin{align}\label{eq:Zakai:SPDE:X:SPDE}
    dq_X^h(t) - L_X^{h*}q_X^h(t)dt - L_Y^{h*}q_X^h(t) dY_t = 0
\end{align}
where the operators $L_X^{h*}$ and $L_Y^{h*}$ are the adjoints of the operators $L_X^{h}$ and $L_Y^{h}$ respectively, and thus defined elementwise as:
\begin{align}\label{eq:operator:Lh:X:adjoint}
& L_X^{h*}f(t,x;y) 
                        \nonumber\\
    &:= \frac{1}{2} \sum_{i,j=1}^n  \frac{\partial^2 }{\partial x_i \partial x_j} \left( \left\{ \Lambda\Lambda'(t,x;y)\right\}_{ij} f(t,x;y) \right)
                        \nonumber\\
    &- \sum_{i=1}^n \frac{\partial }{\partial x_i}\left( \left\{b(t,x;y)
        - \theta \Lambda(t,x;y) \left(\Sigma^{(1)'}(t;y)\ h - \Xi'(t;y) \right)
        \right\}_i f(t,x;y) \right)
\end{align}
and
\begin{align}\label{eq:operator:Lh:Y:adjoint}    
& \left\{L_Y^{h*} f(t,x;y)\right\}_i 
                        \nonumber\\
    &:= \left\{\left[ a^Y(t, x; y) - \theta \Sigma^Y\left(t; y\right) \left(\Sigma^{(1)'}\left(t; y\right) h - \Xi'\left(t; y\right) \right)\right]' 
        \left( \Sigma^Y\Sigma^{Y'}\left(t, y\right)\right)^{-1}\right\}_i f(t,x;y)
                       \nonumber\\         
    &- \sum_{j=1}^n \frac{\partial }{\partial x_j} \left( \left\{\Lambda(t,x;y)\Sigma^{Y'}\left(t; y\right) \left(\Sigma^Y \Sigma^{Y'}\left(t, y\right)\right)^{-1}\right\}_{ji} f(t,x;y) \right),                   
    \quad i = 1, \ldots, m^Y.
\end{align}
\end{remark}

\begin{remark}\label{rk:SPDE:uniqueness}
\cite{na00} addresses the important questions of the existence and uniqueness of the solution in two short remarks, located respectively after the proofs of Proposition 2.1 and Theorem 2.1. We also refer the reader to Sections 4.5 and 4.2.2 in \cite{Bensoussan1992} for a thorough treatment of these questions. Section 4.2.2 therein formally constructs the functional space in which a solution is sought. Theorem 4.2.1. then proves the existence and uniqueness of a solution in this functional space under the assumption that the observation and factor noise are uncorrelated. Section 4.5 extends these results to account for correlation and is, therefore, particularly relevant to our treatment. 
\end{remark}

So far the filter for $X_t$ did not intervene. Nagai and Peng show that in the linear-Gaussian case with the Kalman filter,
one can obtain an explicit analytic solution to the MZE in \eqref{eq:Zakai:SPDE:X} that exploits the specificity of the Kalman Filter. Their result can be adapted to the more general case of nonlinear, but Gaussian, dynamics where the filter is the Kalman filter applied after linearizing the nonlinear coefficients (\textit{Extended Kalman Filter}).

\section{Proof of Proposition \ref{prop:GKSF}}\label{app:proof:GKSF}

First, we show that $\bar{\mathbf{E}} \left[\Psi^X_t \mid \mathcal{F}^Y_t \right] \neq 0$ $\mathbb{P}$-a.s. so the right-hand side of \eqref{eq:GKSF} is well-defined. To see this, notice that $\Psi^X_t \geq 0$ and
$
\bar{\mathbf{E}} \left[1_{\left\{ \Psi^X_t = 0 \right\}} \Psi^X_t \right] =0
$. Recall that $\Psi^X_t = \frac{d\mathbb{P}^h}{d\bar{\mathbb{P}}}\Big{|}_{\mathcal{F}_t}$, then 
$
    \bar{\mathbf{E}} \left[1_{\left\{ \Psi^X_t = 0 \right\}} \Psi^X_t \right] 
    = \mathbf{E}^h \left[1_{\left\{ \Psi^X_t = 0 \right\}} \right]
    = \mathbb{P}^h\left[ \Psi^X_t  = 0 \right],
$
so
$
    \mathbb{P}^h\left[ \Psi^X_t  = 0 \right] = 0.
$
Hence, $\Psi^X_t > 0$ $\mathbb{P}^h$-a.s. which proves our assertion.

Next, we prove \eqref{eq:GKSF}, or equivalently,
\begin{align*}
   & \bar{\mathbf{E}} \left[ \varphi(t,X_t;Y_t) e^{\theta \int_0^t    
     \mathbf{E}^h \left[e^{\theta \int_0^t g(s, X_s, Y_s, h_s;\theta) ds} \mid \mathcal{F}^Y_t \right] ds} \Psi^X_t \mid \mathcal{F}^Y_t \right]\\
    &\hspace{4cm}= 
    \mathbf{E}^h \left[ \varphi(t,X_t;Y_t) e^{\theta \int_0^t g(s, X_s, Y_s, h_s;\theta) ds} \mid \mathcal{F}^Y_t \right]
    \bar{\mathbf{E}} \left[\Psi^X_t \mid \mathcal{F}^Y_t \right]
    \quad \bar{\mathbb{P}}-a.s.
\end{align*}

Both sides are $\mathcal{F}^Y_t$-measurable, so this is equivalent to showing that for any bounded, $\mathcal{F}^Y_t$-measurable random variable $\xi_t$,
\begin{align*}
   & \bar{\mathbf{E}} \left[ \mathbf{E}^h \left[ \varphi(t,X_t;Y_t) e^{\theta \int_0^t     \mathbf{E}^h \left[e^{\theta \int_0^t g(s, X_s, Y_s, h_s;\theta) ds} \mid \mathcal{F}^Y_t \right] ds} \mid \mathcal{F}^Y_t \right]
    \bar{\mathbf{E}} \left[\Psi^X_t \mid \mathcal{F}^Y_t \right] \xi_t \right]\\
&\hspace{2cm}=\bar{\mathbf{E}} \left[ \bar{\mathbf{E}} \left[ \varphi(t,X_t;Y_t) e^{\theta \int_0^t     \mathbf{E}^h \left[ e^{\theta \int_0^t g(s, X_s, Y_s, h_s;\theta) ds} \mid \mathcal{F}^Y_t \right] ds} \Psi^X_t \mid \mathcal{F}^Y_t \right] \xi_t \right]
\end{align*}

To show this equality, we start under the measure $\mathbb{P}^h$. By the Tower property,
\begin{align*}
    &\mathbf{E}^h \left[
        \mathbf{E}^h \left[ \varphi(t,X_t;Y_t) e^{\theta \int_0^t     \mathbf{E}^h \left[  e^{\theta \int_0^t g(s, X_s, Y_s, h_s;\theta) ds} \mid \mathcal{F}^Y_t \right] ds} \mid \mathcal{F}^Y_t \right]
        \xi_t    
    \right]\\
    &\hspace{2cm}= 
    \mathbf{E}^h \left[
        \varphi(t,X_t;Y_t) e^{\theta \int_0^t     \mathbf{E}^h \left[ e^{\theta \int_0^t g(s, X_s, Y_s, h_s;\theta) ds} \mid \mathcal{F}^Y_t \right] ds} \xi_t
    \right]
\end{align*}
Writing this relation under the measure $\bar{\mathbb{P}}$, we get
\begin{align*}
    & \bar{\mathbf{E}} \left[
        \mathbf{E}^h \left[ \varphi(t,X_t;Y_t) e^{\theta \int_0^t     \mathbf{E}^h 
        \left[  e^{\theta \int_0^t g(s, X_s, Y_s, h_s;\theta) ds} \mid \mathcal{F}^Y_t \right] ds} \mid \mathcal{F}^Y_t \right]
        \xi_t \Psi^X_t     \right]\\
   &\hspace{2cm} = 
    \bar{\mathbf{E}} \left[
        \varphi(t,X_t;Y_t) e^{\theta \int_0^t     \mathbf{E}^h \left[e^{\theta \int_0^t g(s, X_s, Y_s, h_s;\theta) ds} \mid \mathcal{F}^Y_t \right] ds}
        \xi_t \Psi^X_t    
    \right].
\end{align*}

By the Tower Property and the assumption that $\xi_t$ is $\mathcal{F}^Y_t$-measurable, we conclude that 
\begin{align*}
    & \bar{\mathbf{E}} \left[
        \mathbf{E}^h \left[ \varphi(t,X_t;Y_t) e^{\theta \int_0^t     \mathbf{E}^h \left[ e^{\theta \int_0^t g(s, X_s, Y_s, h_s;\theta) ds} \mid \mathcal{F}^Y_t \right] ds} \mid \mathcal{F}^Y_t \right] 
        \bar{\mathbf{E}} \left[\Psi^X_t \mid \mathcal{F}^Y_t \right]
        \xi_t 
    \right]\\
  &\hspace{2cm}  = 
    \bar{\mathbf{E}} \left[
        \bar{\mathbf{E}} \left[ \varphi(t,X_t;Y_t) e^{\theta \int_0^t     \mathbf{E}^h \left[  e^{\theta \int_0^t g(s, X_s, Y_s, h_s;\theta) ds} \mid \mathcal{F}^Y_t \right] ds} \Psi^X_t \mid \mathcal{F}^Y_t \right]
        \xi_t
    \right],
\end{align*}
which proves the result $\bar{\mathbb{P}}$-a.s.

\section{Proof of Theorem \ref{theo:Zakai:Z}}\label{app:proof:Zakai:Z}

Inspired by \citet{na00} we start from the Ito differential
{\small
\begin{align*}
& d\left( e^{\theta \int_0^t \hat g_s ds} \Psi_t \varphi(t,\zeta_t;Y_t) \right)
                                            \nonumber\\
= \phantom{} & d\left(e^{\theta \int_0^t \hat g_s ds}\right) \Psi_t \varphi(t,\zeta_t;Y_t)
+   e^{\theta \int_0^t g_s ds} d\Psi_t \varphi(t,\zeta_t;Y_t)
+   e^{\theta \int_0^t \hat g_s ds} \Psi_t d\varphi(t,\zeta_t;Y_t)
+   e^{\theta \int_0^t \hat g_s ds} d \langle \Psi, \varphi \rangle_t                               
                                                \nonumber\\
= \phantom{} &  \theta  e^{\theta \int_0^t \hat g_s ds} \Psi_t \varphi(t,\zeta_t;Y_t) \,\hat g_t\, dt
+  e^{\theta \int_0^t \hat g_s ds} \Psi_t \varphi(t,\zeta_t;Y_t) \check{a}^{Y'}(t,Z_t;h_t) \left( \Sigma^Y\Sigma^{Y'}\left(t, Y_t\right)\right)^{-1} dY_t
                                                \nonumber\\
&+ e^{\theta \int_0^t \hat g_s ds} \Psi_t \left( \frac{\partial \varphi}{\partial t}(t,\zeta_t;Y_t) + \bar{L}\varphi(t,\zeta_t;Y_t) \right)dt
+ e^{\theta \int_0^t \hat g_s ds} \Psi_t D'\varphi(t,\zeta_t;Y_t) H'(t, Z_t) dY_t
                                \nonumber\\
&+   e^{\theta \int_0^t \hat g_s ds} \Psi_t \check{a}^{Y'}(t,Z_t;h_t) H'(t, Z_t) D\varphi(t,\zeta_t;Y_t) dt
\end{align*}
}

where $D\varphi = \frac{\partial \varphi}{\partial \zeta}$. Integrating and taking the expectation, we get
{\small
\begin{align*}
&   \bar{\mathbf{E}} \left[  e^{\theta \int_0^t g_s ds} \Psi_t \varphi(t,\zeta_t;Y_t) \right] 
- \varphi(0,\zeta_0;Y_0)
                                \nonumber\\
=&  \bar{\mathbf{E}} \left[ \int_0^t
 e^{\theta \int_0^s \hat g_u du} \Psi_s \left( 
        \frac{\partial \varphi}{\partial s}(s,\zeta_s;Y_s) 
        + \bar{L}\varphi(s,\zeta_s;Y_s) 
        + \check{a}^{Y'}(s,Z_s;h_s) H'(s, Z_s) D\varphi(t,Z_s)
        + \theta \hat g_s \varphi(s,\zeta_s;Y_s)\right)
ds\right].    
\end{align*}
}

Using Fubini, recalling the definitions of the operator $\hat L^h$, of $\check{a}$, of $\Psi_t$ as well as of $q_Z(\cdot)$, we end up with
{\small
\begin{align*}
q_Z^h(t)(\varphi_t) - q_Z^h(0)(\varphi_0)
&   =\bar{\mathbf{E}} \left[ e^{\theta \int_0^t \hat g_s ds} \Psi_t \varphi(t,\zeta_t;Y_t) \right] 
- \varphi(0,\zeta_0;Y_0)
                                \nonumber\\
&  =\bar{\mathbf{E}}  \left [ \int_0^t
[ e^{\theta \int_0^s \hat g_u du} \Psi_s \left( 
        \frac{\partial \varphi}{\partial t}(s,\zeta_s;Y_s) 
        + \hat L^h\varphi(s,\zeta_s;Y_s) 
        + \theta \hat g_s \varphi(s,\zeta_s;Y_s)\right)
ds\right]   \nonumber\\
&  =\int_0^t\bar{\mathbf{E}} \left[ e^{\theta \int_0^s \hat g_u du} \Psi_s \left( 
        \frac{\partial \varphi}{\partial t}(s,\zeta_s;Y_s) 
        + \hat L^h\varphi(s,\zeta_s;Y_s) 
        + \theta \hat g_s \varphi(s,\zeta_s;Y_s)\right) \right] \,ds \nonumber\\
&=  \int_0^t q_Z^h(s)\left( 
        \frac{\partial \varphi}{\partial t}(s,\zeta_s;Y_s) 
        + \hat L^h\varphi(s,\zeta_s;Y_s) 
        + \theta \hat g(s,Z_s,h_s;\theta) \varphi(s,\zeta_s;Y_s)\right) ds.
\end{align*}
}

%
%


\end{document}